\documentclass{article}
\usepackage{psfrag}
\usepackage[parfill]{parskip}
\usepackage{float, graphicx}
\usepackage[]{epsfig}
\usepackage{amsmath, amsthm, amssymb,}
\usepackage{epsfig}
\usepackage{verbatim}
\usepackage{multicol}
\usepackage{url}
\usepackage{latexsym}
\usepackage{mathrsfs}
\usepackage[colorlinks, bookmarks=true]{hyperref}
\usepackage{amsmath}
\usepackage{enumerate}
\usepackage[normalem]{ulem}
\usepackage{bm}
\usepackage{dsfont}
\usepackage{stmaryrd}
\usepackage{enumitem}
\usepackage{cite}
\usepackage{color}
\usepackage{xcolor}

\makeatletter

\def\@settitle{\begin{center}%
    \bfseries
 \normalfont\LARGE\@title
  \end{center}%
}
\def\@setauthors{\begin{center}%
 \normalsize\@author
  \end{center}%
}

\makeatother

\numberwithin{equation}{section}

\renewcommand{\cal}{\mathcal}
\newcommand\cA{{\mathcal A}}

\newcommand{\cE}{{\cal E}}
\newcommand{\cG}{{\cal G}}

\newcommand{\cL}{{\cal L}}
\newcommand{\cM}{{\cal M}}
\newcommand{\cN}{{\cal N}}

\newcommand{\cP}{{\cal P}}

\newcommand{\cS}{{\mathcal S}}
\newcommand{\cU}{{\mathcal U}}

\newcommand{\fb}{{\mathfrak b}}
\newcommand{\fc}{{\mathfrak c}}
\newcommand{\fd}{{\mathfrak d}}

\newcommand{\fn}{{\mathfrak n}}
\newcommand{\fm}{{\mathfrak m}}

\newcommand{\bma}{{\bm{a}}}
\newcommand{\bmb}{{\bm{b}}}

\newcommand{\bme}{{\bm{e}}}

\newcommand{\bmn}{{\bm{n}}}

\newcommand{\bmr}{{\bm{r}}}
\newcommand{\bmt} {{\bm t }}
\newcommand{\bmu}{{\bm{u}}}
\newcommand{\bmv}{{\bm{v}}}
\newcommand{\bmw}{{\bm{w}}}
\newcommand{\bmx}{{\bm{x}}}

\newcommand{\bms}{\bm s}
\newcommand{\bmm}{{\bm m}}
\newcommand{\bmmu}{{\bm \mu}}

\newcommand{\rd}{{\rm d}}

\newcommand{\ri}{\mathrm{i}}

\newcommand{\bF}{{\mathbb F}}
\newcommand{\bE}{\mathbb{E}}

\newcommand{\bP}{\mathbb{P}}

\newcommand{\bR}{{\mathbb R}}
\newcommand{\bS}{\mathbb S}

\newcommand{\bZ}{\mathbb{Z}}

\DeclareMathOperator{\Tr}{Tr}

\DeclareMathOperator{\vol}{vol}

\DeclareMathOperator{\OO}{O}
\DeclareMathOperator{\oo}{o}
\DeclareMathOperator{\argmax}{argmax}
\DeclareMathOperator{\argmin}{argmin}

\renewcommand{\Re}{\mathop{\mathrm{Re}}}

\newcommand{\deq}{\mathrel{\mathop:}=} 
\newcommand{\eqd}{=\mathrel{\mathop:}} 
\renewcommand{\leq}{\leqslant}
\renewcommand{\geq}{\geqslant}

\newcommand{\beq}{\begin{equation}}
\newcommand{\eeq}{\end{equation}}

\theoremstyle{plain} 
\newtheorem{theorem}{Theorem}[section]
\newtheorem*{theorem*}{Theorem}

\newtheorem*{lemma*}{Lemma}

\newtheorem*{corollary*}{Corollary}
\newtheorem{proposition}[theorem]{Proposition}
\newtheorem*{proposition*}{Proposition}

\newtheorem*{assumption*}{Assumption}

\newtheorem*{definition*}{Definition}

\newtheorem*{example*}{Example}
\newtheorem{remark}[theorem]{Remark}

\newtheorem*{remark*}{Remark}
\newtheorem*{remarks*}{Remarks}

\def\author#1{\par
    {\centering{\authorfont#1}\par\vspace*{0.05in}}
}

\def\titlefont{\fontsize{13}{15}\bfseries\boldmath\selectfont\centering{}}
\def\authorfont{\fontsize{13}{15}}

\let\affiliationfont\rhfont

\def\address#1{\par
    {\centering{\affiliationfont#1\par}}\par\vspace*{11pt}
}

\def\body{
\setcounter{footnote}{0}
\def\thefootnote{\alph{footnote}}
\def\@makefnmark{{$^{\rm \@thefnmark}$}}
}

\def\title#1{
    \thispagestyle{plain}
    \vspace*{-14pt}
    \vskip 79pt
    {\centering{\titlefont #1\par}}%
    \vskip 1em
}

\newcommand{\Mod}[1]{\ (\mathrm{mod}\ #1)}
\newcommand{\spn}{\mathrm{span}}

\begin{document}

\title{Invertibility of adjacency matrices for random $d$-regular  graphs}

\vspace{1.2cm}

 \author{Jiaoyang Huang}
\address{Harvard University\\
   E-mail: jiaoyang@math.harvard.edu}

~\vspace{0.3cm}

\begin{abstract}
Let $d\geq 3$ be a fixed integer and $A$ be the adjacency matrix of a random $d$-regular directed or undirected graph on $n$ vertices. We show there exist constants $\fd>0$,
\begin{align*}
\bP(\text{$A$ is singular in $\bR$})\leq n^{-\fd},
\end{align*}
for $n$ sufficiently large. This answers an open problem by Frieze \cite{MR3728474} and Vu \cite{MR2432537, MR3727622}. The key idea is to study the singularity probability of adjacency matrices over a finite field $\bF_p$. The proof combines a local central limit theorem and a large deviation estimate.
\end{abstract}

\section{Introduction}
The most famous combinatorial problem concerning random matrices is perhaps the ``singularity'' problem. In a standard setting, when the entries of the $n\times n$ matrix are i.i.d. Bernoulli random variables (taking values $\pm1$ with probability $1/2$), this problem was first done by Koml{\'o}s \cite{MR0221962,MR0238371}, where he showed the probability of being singular is $\OO(n^{-1/2})$.  This bound was significantly improved by
Kahn, Koml{\'o}s and Szemer{\'e}di \cite{MR1260107} to an exponential bound
\begin{align*}
\bP(\text{random Bernoulli matrix is singular})< c^n,
\end{align*}
for $c=0.999$, for $c=3/4+\oo(1)$ by Tao and Vu \cite{MR2291914}, and by Rudelson and Vershynin \cite{MR2407948} . The often conjectured optimal value of
$c$ is $1/2 + \oo(1)$, and the best known value $c=1/\sqrt{2} + \oo(1)$ is due to Bourgain, Vu and Wood  \cite{MR2557947}.
Analogous results on singularity of symmetric Bernoulli matrices
were obtained in \cite{MR3158627, MR2891529,MR2267289}.

The above question can be reformulated for the
adjacency matrices of random graphs, either directed
or undirected. Both directed and undirected graphs are abundant in real life. One of the widely studied model in the undirected random graph literature is the Erd{\H o}s-R{\'e}nyi 
graph $G(n,p)$. It was shown by Costello and Vu in \cite{MR2446482}, that the adjacency matrix of $G(n,p)$ is nonsingular with high probability whenever the edge connectivity probability $p$ is above the connectivity threshold $\ln n/n$. For directed  Erd{\H o}s-R{\'e}nyi
graph, a quantitative estimate on the smallest singular value was obtained by Basak and Rudelson in \cite{MR3620692, BR}.

Another intensively studied random graph model is the random $d$-regular graph. For the adjacency matrix of random $d$-regular graphs, its entries are no longer independent. The lack of independence poses significant difficulty for the singularity problem of random $d$-regular graphs. For undirected random $d$-regular graphs, when $d\geq n^{c}$ with any $c>0$, it follows from the bulk universality result \cite{fix2} by Landon, Sosoe and Yau, the adjacency matrix is nonsingular with high probability. For random $d$-regular directed graphs, it was first proven by Cook in \cite{MR3602844}, the adjacency matrix is nonsingular with high probability when $C\ln^2 n\leq d\leq n-C\ln^2 n$. Later in \cite{MR3545253}, it was proven by Litvak, Lytova, Tikhomirov, Tomczak-Jaegermann and Youssef that, when $C\leq d\leq n/(C\ln^2 n)$, the singularity probability is bounded by $\OO(\ln^3 d/\sqrt{d})$. Quantitative estimates on the smallest singular values were derived in \cite{Nik1, basak2018, Alex1}.

For random $d$-regular graphs, the most challenging case is when $d$ is a constant. It was posted as an open problem first appeared in \cite[Conjecture 8.4]{MR2432537} by Vu, and later collected in \cite[Section 9, Problem 7]{MR3728474} by Frieze  and \cite[Conjecture 5.8]{MR3727622} by Vu.
In \cite{Alex2}, it was proven by Litvak, Lytova, Tikhomirov, Tomczak-Jaegermann and Youssef that the adjacency matrix of random $d$-regular directed graphs has rank at least $n-1$ with high probability. In this paper we prove that the adjacency matrix of random $d$-regular directed and undirected graphs is nonsingular with high probability. 

One  approach to estimate the singularity probability of random matrices is to decompose the null vectors $\bS^{n-1}$ into subsets according to different structural properties, e.g., combinatorial dimension \cite{MR1260107,MR2291914}, compressible and imcompressible vectors \cite{MR2407948, Nik1,MR3602844,basak2018,MR2569075}, and statistics of jumps \cite{MR3620692, BR,MR3545253,Alex1,Alex3}.  Different from previous works, which directly study the singularity probability over $\bR$, the key new idea in this paper is to study the singularity probability of adjacency matrices over a finite field $\bF_p$. At first glance, this may seem wasteful, as we discard a great amount of information. Moreover, as a matrix over $\bF_p$, the determinant of the adjacency matrix takes value in $\bF_p$. One expects that the determinant takes value zero with probability about $1/p$. In other words, the adjacency matrix over $\bF_p$ may be singular with positive probability. However, the benefit is that, over finite field $\bF_p$ we can better understand the arithmetic structure of the null vectors, which enables us to obtain a sharp estimate of the singularity probability. We decompose the null vectors $\bF_p^n$ into two classes, the equidistributed class where each number has approximately the same density, and the non-equidistributed class. We estimate the number of adjacency matrices which have a null vector in the equidistributed class using a local central limit theorem, and the number of adjacency matrices which have a null vector in the non-equidistributed class using a large deviation estimate.  In \cite{FLMS}, Ferber, Luh, McKinley and  Samotij use a similar idea to prove resilience results for random Bernoulli matrices.

After the appearance of the current preprint, the asymptotic nonsingularity of adjacency
matrices of random $d$-regular directed graphs and random $d$-regular undirected graphs with even number of vertices are proven by M{\'e}sz{\'a}ros \cite{AM}, and by 
Nguyen and Wood \cite{NW}. The work of M{\'e}sz{\'a}ros \cite{AM} studies the distribution of the sandpile group of random $d$-regular graphs, and determines the distribution of $p$-Sylow subgroup of the sandpile group. Based on \cite{AM}, Nguyen and Wood in \cite{NW}, study the distribution of the cokernels of adjacency matrices of random $d$-regular graphs, and observe that the convergence of such distributions implies asymptotic nonsingularity of the matrices. Our model is slightly more general, and we obtain quantitative estimates on the singularity probability.

\noindent\textbf{Acknowledgement.} I am thankful to Elchanan Mossel and Mustazee Rahman for suggesting the problem of studying adjacency matrices of random $d$-regular graphs over finite fields. I am also grateful to Weifeng Sun for enlightening discussions, and to Nicholas Cook, Van Vu and Melanie Wood for helpful comments on the first draft of this paper.

\subsection{Main results}

We study the configuration model of random $d$-regular directed and undirected graphs, introduced by Bollob{\'a}s in \cite{MR595929} (ideas similar to the configuration model were also presented in \cite{MR0505796,MR545196,MR527727}). By a contiguity argument, our main results also hold for other random $d$-regular directed and undirected graph models, e.g. the uniform model, the sum of $d$ random permutations and the sum of $d$ random perfect matching matrices. 

For the configuration model, one generates a random $d$-regular directed graph by the following procedure:
\begin{enumerate}
\item Associate to each vertex $k \in \{1,2,\cdots,n\}$ a fiber $F_k$ of $d$ points, so that there are
$
\left| \cup_{k\in \{1,2,\cdots,n\}}F_k\right|=nd
$
points in total.
\item Select a permutation $\cP$ of the $nd$ points uniformly at random.
\item For any vertex $k \in \{1,2,\cdots,n\}$, and point $k'\in F_k$, we add a directed edge from vertex $k$ to vertex $\ell$ if the point $\cP(k')$ belongs to fiber $F_{\ell}$.
\end{enumerate}
We denote the $d$-regular directed graphs obtained from the above procedure by $\mathsf{M}_{n,d}$, which is a multiset. It is easy to see from the construction procedure that $|\mathsf{M}_{n,d}|=(nd)!$. Let $\cG\in \mathsf{M}_{n,d}$, one may identify $\cG$ with a random $d$-regular bipartite graph on $n + n$ vertices in the obvious way. We denote $A$ = $A(\cG)$ the adjacency matrix of $\cG$, i.e. $A_{k\ell}$ is the number of directed edges from vertex $k$ to vertex $\ell$.

For the configuration model, one generates a random $d$-regular undirected graph on $n$ vertices with $2|dn$, by the following procedure:
\begin{enumerate}
\item Associate to each vertex $k \in \{1,2,\cdots,n\}$ a fiber $F_k$ of $d$ points, so that there are
$
\left| \cup_{k\in \{1,2,\cdots,n\}}F_k\right|=nd
$
points in total.
\item Select a pairing $\cP$ of the $nd$ points uniformly at random, and add an edge from point $k'$ to point $\ell'$ if $\{k',\ell'\}\in\cP$. 
\item Collapse each fiber $F_k$ to the associated vertex $k$.   
\end{enumerate}
We denote the $d$-regular undirected graphs obtained from the above procedure by $\mathsf{G}_{n,d}$, which is a multiset. It is easy to see from the construction procedure that $|\mathsf{G}_{n,d}|=2^{-nd/2}(nd)!/(nd/2)!$. Let $\cG\in \mathsf{G}_{n,d}$, we denote $A$ = $A(\cG)$ the adjacency matrix of $\cG$, i.e. $A_{k\ell}$ is the number of edges between vertex $k$ and vertex $\ell$.

\begin{theorem}\label{thm:smcount}
Let $d\geq 3$ be a fixed integer, and a prime number $p$ such that $\gcd(p,d)=1$. Then over $\bF_p$, 
we have for $p\ll n^{\min\{1/4, (d-2)/2d\}}$
\begin{align}\label{e:smcountd}
\sum_{\bmv\in \bF_p^n\setminus{\bm0}}|\{\cG\in \mathsf{M}_{n,d}: A(\cG)\bmv=\bm0\}|= (1+\oo(1))|\mathsf{M}_{n,d}|,
\end{align}
and for $p\ll n^{\min\{1/8, (d-2)/(5d-6)\}}$
\begin{align}\label{e:smcountu}
\sum_{\bmv\in \bF_p^n\setminus{\bm0}}|\{\cG\in \mathsf{G}_{n,d}: A(\cG)\bmv=\bm0\}|= (1+\bm1(p=2)+\oo(1))|\mathsf{G}_{n,d}|,
\end{align}
for $n$ sufficiently large.
\end{theorem}


If an adjacency matrix $A(\cG)$ is singular as a matrix in $\bF_p$, then we have
\begin{align*}
|\{\bmv\in \bF_p^n\setminus{\bm0}: A(\cG)\bmv=\bm0 \}|\geq p-1.
\end{align*} 
Therefore it follows from Theorem \ref{thm:smcount},
\begin{align*}
&(p-1)|\{\cG\in \mathsf{M}_{n,d}: A(\cG) \text{ is singular in } \bF_p\}|
\leq \sum_{\bmv\in \bF_p^n\setminus{\bm0}}|\{\cG\in \mathsf{M}_{n,d}: A(\cG)\bmv=\bm0\}|
= (1+\oo(1))|\mathsf{M}_{n,d}|,\\
&(p-1)|\{\cG\in \mathsf{G}_{n,d}: A(\cG) \text{ is singular in } \bF_p\}|
\leq \sum_{\bmv\in \bF_p^n\setminus{\bm0}}|\{\cG\in \mathsf{G}_{n,d}: A(\cG)\bmv=\bm0\}|
= (1+\oo(1))|\mathsf{G}_{n,d}|,
\end{align*}
and we obtain the next theorem.

\begin{theorem}\label{thm:spFp}
Let $d\geq 3$ be a fixed integer, and an odd prime number $p$ such that $\gcd(p,d)=1$. Let $A$ be the adjacency matrix of a random $d$-regular  graph on $n$ vertices. Then as a random matrix in $\bF_p$, 
\begin{align*}
\bP(\text{$A$ is singular in $\bF_p$})\leq \frac{1+\oo(1)}{p-1},
\end{align*}
for $n$ sufficiently large.
\end{theorem}

The entries of $A(\cG)$ are all integers. Therefore, if $A(\cG)$ is singular in $\bR$, it is also singular in any finite field $\bF_p$. The next theorem follows  by taking $p$ large in Theorem \ref{thm:spFp}.

\begin{theorem}\label{thm:spR}
Let $d\geq 3$ be a fixed integer. Let $A$ be the adjacency matrix of a random $d$-regular directed or undirected graph on $n$ vertices. Then there exist constants $\fd>0$, 
\begin{align*}
\bP(\text{$A$ is singular in $\bR$})\leq n^{-\fd},
\end{align*}
for $n$ sufficiently large. 
\end{theorem}

\begin{remark}
For random $d$-regular directed graphs we can take $\fd=\min\{1/4, (d-2)/2d\}$, and for random $d$-regular undirected graphs we can take $\fd=\min\{1/8, (d-2)/(5d-6)\}$.The probability that the adjacency matrix of a random $d$-regular graph is singular is at least polynomial in $1/n$. In fact, if a $d$-regular graph contains the subgraph in Figure \ref{f:K2d},
\begin{figure}
\center
\includegraphics[width=0.6\textwidth]{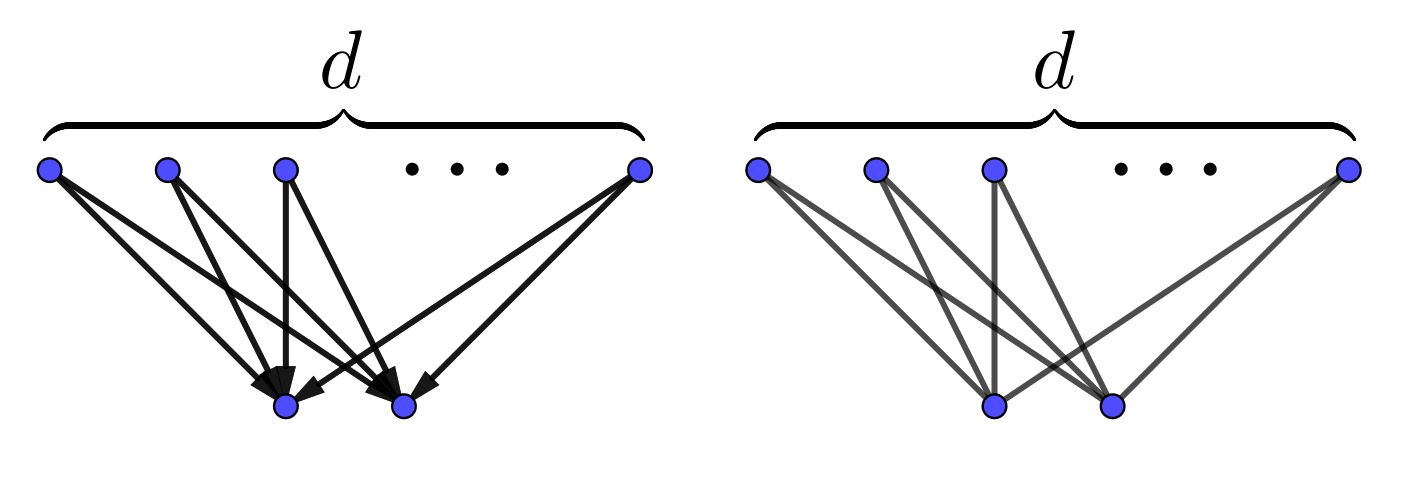}
\caption{If a $d$-regular directed or undirected graph contains the above subgraph, its adjacency matrix is singular.}
\label{f:K2d}
\end{figure}
then its adjacency matrix is singular. As a consequence, it holds that $\bP(\text{$A$ is singular in $\bR$})\geq \OO(1)/n^{d-2}$.
\end{remark}

\begin{remark}
Theorem \ref{thm:spR} for random $d$-regular graphs with even number of vertices was recently proven by M{\'e}sz{\'a}ros \cite{AM}, and by 
Nguyen and Wood \cite{NW}. Their approach studies the distribution of the sandpile group of random $d$-regular graphs.
Although the result in Theorem \ref{thm:spR} is stated for the configuration model, the same statement holds for other models, e.g. the uniform model, the sum of $d$ random permutations and the sum of $d$ random perfect matching matrices,  by a contiguity argument.
\end{remark}

\section{Random Walk Interpretation}
In this section, we enumerate $|\{\cG\in \mathsf{M}_{n,d}: A(\cG)\bmv=\bm0 \text{ in } \bF_p\}|$ and  $|\{\cG\in \mathsf{G}_{n,d}: A(\cG)\bmv=\bm0 \text{ in } \bF_p\}|$ as the number of certain walk paths. Before stating the result, we need to introduce some notations. We define the counting function $\Phi:\cup_{k\geq 1}\bF_p^k\mapsto \bZ^p$, given by 
\begin{align*}
\Phi(a_1,a_2,\cdots, a_k)=\left(\sum_{i=1}^k \bm 1(a_i=0),\sum_{i=1}^k \bm 1(a_i=1), \cdots,\sum_{i=1}^k \bm 1(a_i=p-1)\right).
\end{align*}
We decompose the space $\bF_p^n$ as 
\begin{align*}
\bF_p^n=\bigcup_{n_0,n_1,\cdots, n_{p-1}\in \bZ_{\geq0}\atop n_0+n_1+\cdots+n_{p-1}=n}\cS(n_0,n_1,\cdots, n_{p-1}),
\end{align*}
where 
\begin{align*}
\cS(n_0,n_1,\cdots, n_{p-1})=\{\bmv=(v_1,v_2,\cdots, v_n)\in \bF_p^n: \Phi(\bmv)=(n_0,n_1,\cdots,n_{p-1})\}.
\end{align*}
The cardinality of $\cS(n_0,n_1,\cdots, n_{p-1})$ is 
\begin{align*}
\left|\cS(n_0,n_1,\cdots, n_{p-1})\right|={n\choose n_0,n_1,\cdots,n_{p-1}}.
\end{align*}
%
%
We define the multiset $\cU_{d,p}$
\begin{align}\begin{split}\label{e:defcU}
\cU_{d,p}&=\{\Phi(\bma): \bma=(a_1,a_2,\cdots,a_d)\in \bF_p^d, a_1+a_2+\cdots+a_d=0\}.
\end{split}\end{align}
For any $a_1,a_2,\cdots, a_{d-1}\in \bF_p$, there exists a unique $a_d\in\bF_p$ such that $a_1+a_2+\cdots+a_d=0$. The multiset $\cU_{d,p}$ has cardinality $p^{d-1}$, i.e. $|\cU_{d,p}|=p^{d-1}$. 
We denote $\cM(n_0,n_1,\cdots, n_{p-1})$ the set of $p\times p$ symmetric matrices $M=[m_{ij}]_{0\leq i,j\leq p-1}$, such that
\begin{enumerate}
\item $m_{ij}=m_{ji}\in \bZ_{\geq 0}$ for $0\leq i,j\leq p-1$ and $2|m_{ii}$ for $0\leq i\leq p-1$.
\item $\sum_{j=0}^{p-1}m_{ij}=dn_i$ for $0\leq i\leq p-1$.
\item  $\sum_{j=0}^{p-1}jm_{ij}\equiv 0 \Mod p$ for $0\leq i\leq p-1$.
\end{enumerate}
We denote 
\begin{align*}
\cM=\bigcup_{n_0,n_1,\cdots,n_{p-1}\in \bZ_{\geq 0}, n_0<n\atop n_0+n_1+\cdots+n_{p-1}=n}\cM(n_0, n_1, \cdots, n_{p-1}).
\end{align*}


\begin{proposition}\label{p:walkrepd}
Let $d\geq 3$ be a fixed integer, and $p$ a prime number. Fix $\bmv\in \cS(n_0,n_1,\cdots, n_{p-1})$, we have
\begin{align}\begin{split}\label{e:randomwalkd}
&\phantom{{}={}}|\{\cG\in \mathsf{M}_{n,d}: A(\cG)\bmv=\bm0 \text{ in } \bF_p\}|\\
&=\left(\prod_{j=0}^{p-1} (dn_j)!\right)|\{(\bmu_1,\bmu_2\cdots, \bmu_n)\in \cU_{d,p}^n: \bmu_1+\bmu_2+\cdots+\bmu_n=(dn_0, dn_{1},\cdots, dn_{p-1})\}|\\
&=\left(\prod_{j=0}^{p-1} (dn_j)!\right)p^{n(d-1)}\bP(X_1+X_2+\cdots+X_n=(dn_0, dn_{1},\cdots, dn_{p-1})),
\end{split}\end{align}
where $X_1,X_2,\cdots, X_n$ are independent copies of $X$, which is uniform distributed over $\cU_{d,p}$.
\end{proposition}

\begin{proposition}\label{p:walkrepu}
Let $d\geq 3$ be a fixed integer, and $p$ a prime number. Fix $\bmv\in \cS(n_0,n_1,\cdots, n_{p-1})$, we have
\begin{align}\begin{split}\label{e:randomwalku}
&\phantom{{}={}}|\{\cG\in \mathsf{G}_{n,d}: A(\cG)\bmv=\bm0 \text{ in } \bF_p\}|\\
&=\sum_{M\in \cM(n_0, n_1,\cdots, n_{p-1})}\prod_{0\leq i<j\leq p-1}m_{ij}!\prod_{i=0}^{p-1}\frac{m_{ii}!}{2^{m_{ii}/2}(m_{ii}/2)!}\\
&\times\prod_{i=0}^{p-1}|\{(\bmu_1,\bmu_2\cdots, \bmu_{n_i})\in \cU_{d,p}^{n_i}: \bmu_1+\bmu_2+\cdots\bmu_{n_i}=(m_{i0},m_{i1},\cdots, m_{ip-1}) \}|\\
&=\sum_{M\in \cM(n_0, n_1,\cdots, n_{p-1})}\prod_{0\leq i<j\leq p-1}m_{ij}!\prod_{i=0}^{p-1}\frac{m_{ii}!}{2^{m_{ii}/2}(m_{ii}/2)!}p^{n(d-1)}\\
&\times\prod_{i=0}^{p-1}\bP(X_1+X_2+\cdots+X_{n_i}=(m_{i0}, m_{i1},\cdots, m_{ip-1})),
\end{split}\end{align}
where $X_1,X_2,\cdots, X_{n_i}$ are independent copies of $X$, which is uniform distributed over $\cU_{d,p}$.
\end{proposition}

\begin{proof}[Proof of Proposition \ref{p:walkrepd}]
We recall the configuration model for random $d$-regular directed graphs from the introduction that each vertex $k \in \{1,2,\cdots,n\}$ is associated with a fiber $F_k$ of $d$ points.  For each permutation $\cP$ of the $nd$ points, we associate it a map $f_{\cP}: \cup_{k\in \{1,2,\cdots, n\}}F_k\mapsto \bF_p$ in the following way. For any point $k'$, if $\cP(k')=\ell'$ and $\ell'\in F_\ell$, then $f_{\cP}(k')=v_\ell$. A given map $f: \cup_{k\in \{1,2,\cdots, n\}}F_k\mapsto \bF_p$ is from a permutation if
\begin{align}\label{e:permcond}
\sum_{k\in \{1,2,\cdots, n\}}\sum_{k'\in F_k}\bm1(f(k')=j)=dn_j, \quad i=0,1,\cdots, p-1.
\end{align}
If this is the case, the number of permutation $\cP$ such that $f_{\cP}=f$ is given by
\begin{align}\label{e:numpaird}
\prod_{j=0}^{p-1}(dn_j)!.
\end{align}
Let $\cG\in \mathsf{M}_{n,d}$ corresponding to a permutation $\cP$. $A(\cG)\bmv=\bm 0$ in $\bF_p$ if and only if for any $k\in \{1,2,\cdots, n\}$
\begin{align}\label{e:sum0d}
\sum_{k'\in F_k}f(k')=0.
\end{align}
Especially, $\Phi(f(k'):k'\in F_k) \in \cU_{d,p}$ for any $k\in \{1,2,\cdots, n\}$. The number of maps $f: \cup_{k\in \{1,2,\cdots, n\}}F_k\mapsto \bF_p$ satisfying \eqref{e:permcond} and \eqref{e:sum0d} is given by
\begin{align}\label{e:nummapd}
|\{(\bmu_1,\bmu_2\cdots, \bmu_n)\in \cU_{d,p}^n: \bmu_1+\bmu_2+\cdots+\bmu_n=(dn_0, dn_{1},\cdots, dn_{p-1})\}|.
\end{align}
The claim \eqref{e:randomwalkd} follows from \eqref{e:numpaird} and \eqref{e:nummapd}.
\end{proof}

\begin{proof}[Proof of Proposition \ref{p:walkrepu}]
We recall the configuration model for random $d$-regular undirected graphs  from the introduction that each vertex $k \in \{1,2,\cdots,n\}$ is associated with a fiber $F_k$ of $d$ points.  For each pairing $\cP$ of the $nd$ points, we associate it a map $f_{\cP}: \cup_{k\in \{1,2,\cdots, n\}}F_k\mapsto \bF_p$ in the following way. For any point $k'$, if $\{k',\ell'\}\in \cP$ and $\ell'\in F_\ell$, then $f_{\cP}(k')=v_\ell$. Given a map $f: \cup_{k\in \{1,2,\cdots, n\}}F_k\mapsto \bF_p$, we denote its data matrix as
\begin{align*}
M_f=[m_{ij}]_{0\leq i,j\leq p-1}, \quad m_{ij}=\sum_{k\in \{1,2,\cdots, n\}}\sum_{k'\in F_k}\bm1(v_k=i)\bm1(f(k')=j).
\end{align*}
The map $f: \cup_{k\in \{1,2,\cdots, n\}}F_k\mapsto \bF_p$ is from a pairing $\cP$ if for any $\{k', \ell'\}\in \cP$ with $k'\in F_k$ and $\ell'\in F_\ell$, it holds $f(k')=v_{\ell}$ and $f(\ell')=v_k$. This is possible, if and only if its data matrix $M_f=[m_{ij}]_{0\leq i,j\leq p-1}$ satisfies
\begin{enumerate}
\item $m_{ij}=m_{ji}\in \bZ_{\geq 0}$ for $0\leq i,j\leq p-1$ and $2|m_{ii}$ for $0\leq i\leq p-1$.
\item $\sum_{j=0}^{p-1}m_{ij}=dn_i$ for $0\leq i\leq p-1$.
\end{enumerate}
If this is the case, the number of pairings $\cP$ such that $f_{\cP}=f$ is given by
\begin{align}\label{e:numpair}
\prod_{0\leq i<j\leq p-1}m_{ij}!\prod_{i=0}^{p-1}\frac{m_{ii}!}{2^{m_{ii}/2}(m_{ii}/2)!}.
\end{align}
Let $\cG\in \mathsf{G}_{n,d}$ corresponding to a pairing $\cP$. $A(\cG)\bmv=\bm 0$ in $\bF_p$ if and only if for any $k\in \{1,2,\cdots, n\}$
\begin{align}\label{e:sum0}
\sum_{k'\in F_k}f(k')=0.
\end{align}
Especially, $\Phi(f(k'):k'\in F_k) \in \cU_{d,p}$ for any $k\in \{1,2,\cdots, n\}$. If this is the case, we have
\begin{align*}
0=\sum_{k\in \{1,2,\cdots, n\}}\bm1(v_k=i)\sum_{k'\in F_k}f(k')\equiv \sum_{j=0}^{p-1}jm_{ij} \Mod p,
\end{align*}
and thus the data matrix $M_f\in \cM(n_0,n_1,\cdots, n_{p-1})$. Given a data matrix $M\in \cM(n_0,n_1,\cdots, n_{p-1})$, the number of maps $f: \cup_{k\in \{1,2,\cdots, n\}}F_k\mapsto \bF_p$ with data matrix $M$ satisfying \eqref{e:sum0} is given by
\begin{align}\label{e:nummap}
\prod_{i=0}^{p-1}|\{(\bmu_1,\bmu_2\cdots, \bmu_{n_i})\in \cU_{d,p}^{n_i}: \bmu_1+\bmu_2+\cdots\bmu_{n_i}=(m_{i0},m_{i1},\cdots, m_{ip-1}) \}|.
\end{align}
The claim \eqref{e:randomwalku} follows from \eqref{e:numpair} and \eqref{e:nummap}.
\end{proof}

Let $X$ be a random vector uniform distributed over $\cU_{d,p}$. The mean of $X$ is given by
\begin{align}\label{e:mean}
\bE[X(j)]
=\frac{1}{p^{d-1}}\sum_{(a_1,a_2,\cdots,a_d)\in \bF_p^d\atop a_1+a_2+\cdots+a_d=0}\sum_{k=1}^d\bm1(a_k=j)=\frac{d}{p},\quad j=0,1,\cdots,p-1.
\end{align}
The covariance of $X$ is given by
\begin{align}\begin{split}\label{e:cov}
&\phantom{{}={}}\bE[(X(j)-d/p)(X(j')-d/p)]
=\frac{1}{p^{d-1}}\sum_{(a_1,a_2,\cdots,a_d)\in \bF_p^d\atop a_1+a_2+\cdots+a_d=0}\sum_{1\leq k,k'\leq d}\bm1(a_k=j)\bm1(a_{k'}=j')-\frac{d^2}{p^2}\\
&=\frac{1}{p^{d-1}}\sum_{(a_1,a_2,\cdots,a_d)\in \bF_p^d\atop a_1+a_2+\cdots+a_d=0}\left(\delta_{jj'}\sum_{1\leq k\leq d}\bm1(a_k=j)+\sum_{1\leq k\neq k'\leq d}\bm1(a_k=j)\bm1(a_{k'}=j')\right)-\frac{d^2}{p^2}=\frac{d}{p}\delta_{jj'}-\frac{d}{p^2},
\end{split}\end{align}
for any $j,j'=0,1,\cdots,p-1$.
We summarize \eqref{e:mean} and \eqref{e:cov} as
\begin{align}\label{e:meanandcov}
\bmmu\deq \bE[X]=(d/p,d/p,\cdots,d/p),\quad \Sigma\deq\bE[(X-\bmmu)(X-\bmmu)^t]=\frac{d}{p}I_p-\frac{d}{p^2}\bm1\bm1^t.
\end{align}
We denote the characteristic function of $X$ as
\begin{align*}
\phi_{X}(\bmt)=\bE[\exp\{\ri\langle\bmt, X\rangle\}],\quad \phi_{X-\bmmu}(\bmt)=\bE[\exp\{\ri\langle\bmt, X-\bmmu\rangle\}]=e^{-\ri \langle \bmt,\bmmu\rangle}\phi_{X}(\bmt).
\end{align*}
The lattice spanned by vectors in $\cU_{d,p}$ is the dual lattice of $\spn \{(0,1/p, 2/p,\cdots, p-1/p), \bme_1,\bme_2,\cdots, \bme_{p}\}$
in $\{(x_1, x_2,\cdots, x_p)\in \bR^p: x_1+x_2+\cdots+x_p=d\}$, where $\bme_1,\bme_2,\cdots, \bme_{p}$ is the standard base of $\bR^p$. Therefore $|\phi_{X-\bmmu}^n(\bmt)|=1$ if any only if 
\begin{align}\label{e:line}
\bmt\in2\pi(0,1/p,2/p,\cdots, p-1/p)\bZ+(1,1,\cdots,1)\bR.
\end{align}
The following proposition gives a quantitative estimate of the absolute value of the characteristic function $|\phi_{X-\mu}(\bmt)|$ when $\bmt$ is away from these lines \eqref{e:line}. 
Fix small $\delta>0$. For $j=0,1,2,\cdots,p-1$, we define domains 
\begin{align}\label{e:domainB}
B_j(\delta)=2\pi j(0,1/p,2/p,\cdots, (p-1)/p)+Q (\{\bmx\in \bR^{p-1}:\|\bmx\|_2^2\leq \delta\}\times [0,2\sqrt{p}\pi]), 
\end{align}
where $Q$ is an orthogonal transform, given by the $p\times p$ orthogonal matrix  $Q=[O, \bm1/\sqrt{p}]$, and $Q (\{\bmx\in \bR^{p-1}:\|\bmx\|_2^2\leq \delta\}\times [0,2\sqrt{p}\pi])$ is the image of $\{\bmx\in \bR^{p-1}:\|\bmx\|_2^2\leq \delta\}\times [0,2\sqrt{p}\pi]$ under the orthogonal transform $Q$.
\begin{proposition}\label{p:cfbound}
For any $\delta>0$ small enough, and $\bmt\in (2\pi \bR^p/\bZ^p)\setminus \cup_{j=0}^{p-1}B_j(\delta)$, there exists a constant $c(\delta)>0$,
\begin{align}\label{e:cfbound}
|\phi_{X-\bmmu}(\bmt)|\leq 1-c(\delta)/p^3.
\end{align}
\end{proposition}
\begin{proof}
We prove \eqref{e:cfbound} by contradiction: if for some $\bmt$,
\begin{align}\label{e:cfbound2}
\left|\frac{1}{p^{d-1}}\sum_{j=1}^{p^{d-1}}e^{\ri \langle \bmt, \bmw_j\rangle}\right|\geq 1-c(\delta)/p^3,
\end{align}
then $\bmt\in \cup_{j=0}^{p-1}B_j(\delta)$.
Without loss of generality, we can assume (by shifting $\bmt$) that $t_0=0$ and $|t_1|\leq 1/2p$. We denote $\psi=\arg \phi_X(\bmt)$. There exists a subset $\cU'\in \cU_{d,p}$, with $|\cU'|\leq p^{d-2}/(2d)$, such that for any $\bmw \in \cU_{d,p}\setminus \cU'$, it holds 
\begin{align}\label{e:smalldev}
\langle \bmt, \bmw\rangle=2\pi n_\bmw+\psi+\varepsilon_\bmw, \quad |\varepsilon_\bmw|\leq \frac{\varepsilon}{p},
\end{align}
where $1-\cos(\varepsilon/p)=2dc(\delta)/p^2$. Otherwise, we have
\begin{align*}
\left|\frac{1}{p^{d-1}}\sum_{j=1}^{p^{d-1}}e^{\ri \langle \bmt, \bmw_j\rangle}\right|
&=\Re \frac{1}{p^{d-1}}\sum_{j=1}^{p^{d-1}}e^{\ri \langle \bmt, \bmw_j\rangle-\ri\psi}\\
&< \frac{1}{p^{d-1}}\left(|\cU_{d,p}\setminus \cU'|+\sum_{\bmw\in \cU'} \cos(\varepsilon/p)\right)
=1-\frac{c(\delta)}{p^3},
\end{align*}
which contradicts with \eqref{e:cfbound2}. We show next that \eqref{e:smalldev} holds for all $\bmw\in \cU_{d,p}$ with slightly worse error. We consider $d\times d$ arrays in $\bF_p$ such that the sum of each row and column is zero (zero sum $d\times d$ arrays). Fix any $\bma^1=(a^1_1,a^1_2,\cdots, a^1_d)\in \bF_p^d$ with $a^1_1+a^1_2+\cdots+a^1_d=0$. The total number of zero sum $d\times d$ arrays with the first row given by $\bma^1$ is $p^{(d-1)(d-2)}$. For any $\bmb=(b_1,b_2,\cdots, b_d)\in \bF_p^d$ with $b_1+b_2+\cdots+b_d=0$ and $\bmb\neq \bma^1$, the total number of zero sum $d\times d$ arrays with the first row given by $\bma^1$ and one row or column given by $\bmb$ is at most $(d-1)p^{(d-1)(d-3)}+dp^{(d-2)(d-2)}$. Since $p^{(d-1)(d-2)}> ((d-1)p^{(d-1)(d-3)}+dp^{(d-2)(d-2)})|\cU'|$, there exists a zero sum $d\times d$ arrays with rows given by $\bma^1, \bma^2,\cdots, \bma^d$ and columns given by $\bmb^1, \bmb^2,\cdots,\bmb^d$ such that $\Phi(\bma^2), \Phi(\bma^3),\cdots, \Phi(\bma^d), \Phi(\bmb^1), \Phi(\bmb^2),\cdots, \Phi(\bmb^d)\in \cU_{b,p}\setminus \cU'$. As a consequence, we have
\begin{align*}\begin{split}
\langle \bmt, \Phi(\bma^1)\rangle 
&=
\sum_{i=1}^d\langle \bmt, \Phi(\bmb^i)\rangle 
-\sum_{i=2}^d\langle \bmt, \Phi(\bma^i)\rangle\\
&=2\pi \left(\sum_{i=1}^d n_{\Phi(\bmb^i)}-\sum_{i=2}^d n_{\Phi(\bma^i)}\right)+\psi
+\left(\sum_{i=1}^d \varepsilon_{\Phi(\bmb^i)}-\sum_{i=2}^d \varepsilon_{\Phi(\bma^i)}\right)\\
&\eqd2\pi n_{\Phi(\bma^1)}+\psi
+\varepsilon_{\Phi(\bma^1)},
\end{split}\end{align*}
where $|\varepsilon_{\Phi(\bma^1)}|\leq 2d\varepsilon/p$. Especially if we take $\bma^1=(0,0,\cdots,0)$ then $\langle \bmt, \Phi(\bma^1)\rangle=0$, and we get $|\psi|\leq 2d\varepsilon/p$. We can absorb $\psi$ into the error term. Therefore, uniformly for all $\bmw\in \cU_{d,p}$, 
\begin{align}\label{e:twprod}
\langle \bmt, \bmw\rangle=2\pi n_\bmw +\varepsilon_\bmw,\quad  |\varepsilon_\bmw|\leq \frac{4d\varepsilon}{p}.
\end{align}

We take two family of vectors $\bmu_k=(d-2)\bme_0+\bme_k+\bme_{p-k},\bmv_k=(d-3)\bme_0+\bme_1+\bme_{k-1}+\bme_{p-k}\in \cU_{d,p}$ for $k=2,3,\cdots, p-1$. By taking $\bmw=\bmu_k, \bmv_k$ in \eqref{e:twprod}, we get
\begin{align*}
t_k=t_{k-1}+t_1+2\pi(n_{\bmv_k}-n_{\bmu_k})+(\varepsilon_{\bmv_k}-\varepsilon_{\bmu_k}), \quad k=2,3,\cdots,p-1,
\end{align*}
and therefore,
\begin{align}\label{e:diffbound}
t_k=kt_1+2\pi\sum_{j=2}^k (n_{\bmv_j}-n_{\bmu_j})+\sum_{j=2}^k(\varepsilon_{\bmv_j}-\varepsilon_{\bmu_j}).
\end{align}
We can shift each $t_k$ by $2\pi \bZ$, and assume that 
\begin{align*}
t_k=kt_1+\sum_{j=2}^k(\varepsilon_{\bmv_j}-\varepsilon_{\bmu_j}).
\end{align*}
Therefore, thanks to the bound \eqref{e:twprod}, we get $|t_k|\leq k/2p+8(k-1)d\varepsilon/p\leq1$ for all $1\leq k\leq p-1$, provided we take $\varepsilon$ small enough. Let $k_{\max}=\argmax_k t_k$ and $k_{\min}=\argmin_k t_k$. By taking $\bmw_1=(d-2)\bme_0+\bme_{k_{\max}}+\bme_{p-k_{\max}}$ and $\bmw_2=(d-3)\bme_0+\bme_{k_{\min}}+\bme_{k_{\max}-k_{\min}}+\bme_{p-k_{\max}}$ in \eqref{e:twprod}, we get
\begin{align*}
3\geq |t_{k_{\max}}-t_{k_{\min}}-t_{k_{\max}-k_{\min}}|=|2\pi(n_{\bmw_2}-n_{\bmw_1})+(\varepsilon_{\bmw_2}-\varepsilon_{\bmw_1})|\geq 2\pi|n_{\bmw_2}-n_{\bmw_1}|-|\varepsilon_{\bmw_2}-\varepsilon_{\bmw_1}|.
\end{align*}
Therefore, $n_{\bmw_2}=n_{\bmw_1}$ and 
\begin{align*}
|t_{k_{\min}}|\leq |t_{k_{\max}}-t_{k_{\min}}-t_{k_{\max}-k_{\min}}|=|\varepsilon_{\bmw_2}-\varepsilon_{\bmw_1}|\leq \frac{8d\varepsilon}{p}.
\end{align*}
By symmetry, we also have $|t_{k_{\max}}|\leq 8d\varepsilon/p$. It follows that $|t_0|,|t_1|,\cdots, |t_{p-1}|\leq \max\{|t_{k_{\max}}|, |t_{k_{\min}}|\}\leq 8d\varepsilon/p$. So $\bmt\in B_0(\delta)$ and the claim follows, provided we take $\varepsilon$ small enough.
\end{proof}

\section{Proof of Theorem \ref{thm:smcount} for Random Directed $d$-regular Graphs}
Thanks to Proposition \ref{p:walkrepd}, we can rewrite the lefthand side of \eqref{e:smcountd} as
\begin{align*}
\sum_{\bmv\in \bF_p^n\setminus{\bm0}}|\{\cG\in \mathsf{M}_{n,d}: A(\cG)\bmv=\bm0\}|
&=\sum_{n_0,n_1,\cdots,n_{p-1}\in \bZ_{\geq 0}, n_0<n\atop n_0+n_1+\cdots+n_{p-1}=n}\sum_{\bmv\in \cS(n_0,n_1,\cdots, n_{p-1})} |\{\cG\in \mathsf{M}_{n,d}: A(\cG)\bmv=\bm0\}|\\
&=\sum_{n_0,n_1,\cdots,n_{p-1}\in \bZ_{\geq 0}, n_0<n\atop n_0+n_1+\cdots+n_{p-1}=n}{n\choose n_0,n_1,\cdots, n_{p-1}}\left(\prod_{j=0}^{p-1} (dn_j)!\right)\times\\
&\phantom{{}={}}\times p^{n(d-1)}\bP(X_1+X_2+\cdots+X_n=(dn_0, dn_{1},\cdots, dn_{p-1})).
\end{align*}
Therefore Theorem \ref{thm:smcount} is equivalent to the following estiamte
\begin{align}\begin{split}\label{e:keyestd}
&\sum_{n_0,n_1,\cdots,n_{p-1}\in \bZ_{\geq 0}, n_0<n\atop n_0+n_1+\cdots+n_{p-1}=n}{n\choose n_0,n_1,\cdots, n_{p-1}}{dn\choose dn_0,dn_1,\cdots, dn_{p-1}}^{-1}\times\\
&\times p^{n(d-1)}\bP(X_1+X_2+\cdots+X_n=(dn_0, dn_{1},\cdots, dn_{p-1}))=1+\oo(1).
\end{split}\end{align}

To prove \eqref{e:keyestd}, we fix a large number $\fb>0$, and decompose those $p$-tuples $(n_0,n_1,\cdots,n_{p-1})$ into two classes:
\begin{enumerate}
\item (Equidistributed) $\cE$ is the set of $p$-tuples $(n_0,n_1,\cdots,n_{p-1})\in \bZ_{\geq 0}^n$, such that $\sum_{j=0}^{p-1}(n_j/n-1/p)^2\leq \fb \ln n/n$.
\item (Non-equidistributed) $\cN$ is the set of $p$-tuples $(n_0,n_1,\cdots,n_{p-1})\in \bZ_{\geq 0}^n$, which are not $(n,0,0,\cdots, 0)$ or equidistributed.
\end{enumerate}
In Section \ref{subs:lcltd}, we estimate the sum of terms in \eqref{e:keyestd} corresponding to equidistributed $p$-tuples using a local central limit theorem. In Section \ref{subs:ldpd}, we show that the sum of terms in \eqref{e:keyestd} corresponding to non-equidistributed $p$-tuples is small, via a large deviation estimate. Theorem \ref{thm:smcount} for random directed $d$-regular graphs follows from combining Proposition \ref{p:lcltd} and Proposition \ref{p:ldpd}.

\subsection{Local central limit theorem estimate}
\label{subs:lcltd}

In this section, we estimate the sum of terms in \eqref{e:keyestd} corresponding to equidistributed $p$-tuples $(n_0,n_1,\cdots,n_{p-1})$, using a local central limit theorem.
\begin{proposition}\label{p:lcltd}
Let $d\geq 3$ be a fixed integer, and a prime number $p$ such that $\gcd(p,d)=1$ and $p\ll n^{1/4}$. Then for $n$ sufficiently large
\begin{align}\label{e:equd}
\sum_{(n_0,n_1,\cdots,n_{p-1}) \in \cE}\sum_{\bmv\in \cS(n_0,n_1,\cdots, n_{p-1})} |\{\cG\in \mathsf{M}_{n,d}: A(\cG)\bmv=\bm0\}|=\left(1+\OO\left(\frac{p^2(\ln n)^{3/2}}{\sqrt{n}}\right)\right)|\mathsf{M}_{n,d}|.
\end{align}
\end{proposition}

\begin{proof}
Thanks to Proposition \ref{p:walkrepd}, we have
\begin{align}\begin{split}\label{e:explcltd}
&\phantom{{}={}}\frac{1}{|\mathsf{M}_{n,d}|}\sum_{(n_0,n_1,\cdots,n_{p-1}) \in \cE}\sum_{\bmv\in \cS(n_0,n_1,\cdots, n_{p-1})} |\{\cG\in \mathsf{M}_{n,d}: A(\cG)\bmv=\bm0\}|\\
&=\sum_{(n_0,n_1,\cdots,n_{p-1})\in \cE}{n\choose n_0,n_1,\cdots, n_{p-1}}{dn\choose dn_0,dn_1,\cdots, dn_{p-1}}^{-1}\\
&\phantom{{}={}}\times p^{n(d-1)}\bP(X_1+X_2+\cdots+X_n=(dn_0, dn_{1},\cdots, dn_{p-1})),
\end{split}\end{align}
where $X_1,X_2,\cdots, X_n$ are independent copies of $X$, which is uniform distributed over $\cU_{d,p}$ as defined in \eqref{e:defcU}.
For an equidistributed $p$-tuple $(n_0,n_1,\cdots,n_{p-1})$, we denote $\fn_j=n_j/n$ for $j=0,1,\cdots, p-1$. Then by our definition of $\cE$, we have 
$\sum_{j=0}^{p-1}(\fn_j-1/p)^2\leq \fb \ln n/n$. We estimate the first factor on the righthand side of \eqref{e:explcltd} using Stirling's formula,
\begin{align}\begin{split}\label{e:firstfactord}
&\phantom{{}={}}{n\choose n_0,n_1,\cdots, n_{p-1}}{dn\choose dn_0,dn_1,\cdots, dn_{p-1}}^{-1}p^{(d-1)n}\\
&=\left(1+\OO\left(\frac{p^2}{n}\right)\right)d^{\frac{p-1}{2}}
\exp\left\{(d-1)n\left(\sum_{j=0}^{p-1}\fn_j\ln \fn_j +\ln p\right)\right\}\\
&=\left(1+\OO\left(\frac{p(\ln n)^{3/2}}{\sqrt{n}}\right)\right)d^{\frac{p-1}{2}}
\exp\left\{\frac{(d-1)pn}{2}\sum_{j=0}^{p-1}(\fn_j-1/p)^2\right\}.
\end{split}\end{align}

In the following, we estimate $\bP(S_n=(dn_0, dn_{1},\cdots, dn_{p-1}))$, where $S_n=X_1+X_2+\cdots+X_n$.
We recall that $X_1, X_2,\cdots, X_n$ are independent copies of $X$, which is uniformly distributed over the multiset $\cU_{d,p}$. 
For a $p$-tuple $(n_0,n_1,\cdots,n_{p-1})$, if $\sum_{j=0}^{p-1}jn_j\not\equiv 0 \Mod p$, then $\bP(S_n=(dn_0,dn_1,\cdots, dn_{p-1}))=0$. We only need to consider $p$-tuples $(n_0,n_1,\cdots,n_{p-1})$ such that $\sum_{j=0}^{p-1}jn_j\equiv 0 \Mod p$. We denote $\bmn=(n_0,n_1,\cdots, n_{p-1})$. By inverse Fourier formula
\begin{align*}\begin{split}
\bP(S_n=d\bmn)
&=\frac{1}{(2\pi)^p}\int_{2\pi \bR^p/\bZ^p}\phi_X^n(\bmt)e^{-\ri \langle \bmt, d\bmn\rangle}\rd \bmt\\
&=\frac{1}{(2\pi)^p}\int_{2\pi \bR^p/\bZ^p}\phi_{X-\bmmu}^n(\bmt)e^{-\ri \langle \bmt, d\bmn-n\bmmu\rangle}\rd \bmt,
\end{split}\end{align*}
where $\phi_X(\bmt)$ and $\phi_{X-\bmmu}(\bmt)$ are the characteristic functions of $X$ and $X-\bmmu$ respectively. 
We recall the domains $B_j(\delta)$ for $j=0,1,2,\cdots, p-1$ from \eqref{e:domainB}. Thanks to Proposition \ref{p:cfbound}, the characteristic function $|\phi^n_{X-\bmmu}(\bmt)|$ is exponentially small outside those sets $B_j(\delta)$.
\begin{align}\begin{split}\label{e:integ1d}
\bP(S_n=d\bmn)
&=\frac{1}{(2\pi)^p}\sum_{j=0}^{p-1}\int_{2\pi B_j(\delta)}\phi_{X-\bmmu}^n(\bmt)e^{-\ri \langle \bmt, d\bmn-n\bmmu\rangle}\rd \bmt + e^{-c(\delta)n/p^3}\\
&=\frac{p}{(2\pi)^p}\int_{2\pi B_0(\delta)}\phi_{X-\bmmu}^n(\bmt)e^{-\ri \langle \bmt, d\bmn-n\bmmu\rangle}\rd \bmt + e^{-c(\delta)n/p^3},
\end{split}\end{align}
where we used the fact that the integrand is translation invariant by vectors $2\pi(0,1/p, 2/p,\cdots, p-1/p)\bZ$.
For any $\bmt\in  B_0(\delta)$, by definition there exists $\bmx\in \bR^{p-1}$ with $\|\bmx\|^2_2\leq \delta$ and $y\in [0,2\sqrt{p}\pi]$, such that $\bmt=Q(\bmx,y)=O\bmx+(y/\sqrt{p})\bm1$. By a change of variable, we can rewrite \eqref{e:integ1d} as
\begin{align}\begin{split}\label{e:integ2d}
&\phantom{{}={}}\frac{p}{(2\pi)^p}\int_{2\pi B_0(\delta)}\phi_{X-\bmmu}^n(\bmt)e^{-\ri \langle \bmt, d\bmn-n\bmmu\rangle}\rd \bmt\\
&=\frac{p}{(2\pi)^{p}}\int_{\{\bmx\in \bR^{p-1}:\|\bmx\|^2_2\leq \delta\}\times [0,2\sqrt{p}\pi]}\phi_{X-\bmmu}^n(Q(\bmx,y))e^{-\ri \langle Q(\bmx,y), d\bmn-n\bmmu\rangle}\rd \bmx\rd y\\
&=\frac{p^{3/2}}{(2\pi)^{p-1}}\int_{\{\bmx\in \bR^{p-1}:\|\bmx\|^2_2\leq \delta\}}\phi_{X-\bmmu}^n(O\bmx)e^{-\ri \langle O\bmx, d\bmn-n\bmmu\rangle}\rd \bmx,
\end{split}\end{align}
where we used that $\langle \bm1, X-\bmmu\rangle=0$ and $\langle\bm1, d\bmn-n\bmmu\rangle=0$. By Taylor expansion, the characteristic function is
\begin{align}\begin{split}\label{e:cf}
\phi_{X-\bmmu}(O\bmx)
&=\bE\left[1+\ri\langle O\bmx, X-\bmmu\rangle-\frac{1}{2}\langle O\bmx, X-\bmmu\rangle^2-\frac{\ri}{6}\langle O\bmx, X-\bmmu\rangle^3+\OO(\langle O\bmx, X-\mu\rangle^4)\right]\\
&=1-\frac{1}{2}\bmx^tO^t\Sigma O\bmx+\OO\left(\frac{\|\bmx\|_2^3}{p}\right)
=1-\frac{d}{2p}\|\bmx\|_2^2+\OO\left(\frac{\|\bmx\|_2^3}{p}\right),
\end{split}\end{align}
where we used $\Sigma=dI_p/p-d\bm1\bm1^t/p^2$ from \eqref{e:meanandcov}, and $O^t\Sigma O=d I_{p-1}/p$. Fix a large constant $\fc$, which will be chosen later. 
For $\fc p^2\ln n/n\leq \|\bmx\|_2^2\leq \delta$, we have
\begin{align}\label{e:integ3d}
|\phi_{X-\bmmu}(O\bmx)|^n\leq \exp\left\{-\left(\frac{\fc d}{2}+\oo(1)\right)p\ln n\right\},
\end{align}
which turns out to be negligible provided $\fc$ is large enough. In the following we will restrict the integral \eqref{e:integ2d} on the domain 
$\{\bmx\in \bR^{p-1}:\|\bmx\|^2_2\leq \fc p^2\ln n/n\}$. From \eqref{e:cf}, on the domain $\{\bmx\in \bR^{p-1}:\|\bmx\|^2_2\leq \fc p^2\ln n/n\}$, we have 
\begin{align*}
\phi_{X-\bmmu}^n(O\bmx)=\left(1+\OO\left(\frac{p^2(\ln n)^{3/2}}{n^{1/2}}\right)\right)e^{-\frac{dn}{2p}\|\bmx\|^2_2},
\end{align*}
and
\begin{align}\begin{split}\label{e:integ4d}
&\phantom{{}={}}\frac{p^{3/2}}{(2\pi)^{p-1}}\int_{\{\bmx\in \bR^{p-1}:\|\bmx\|^2_2\leq \fc p^2\ln n/n\}}\phi_{X-\bmmu}^n(O\bmx)e^{-\ri \langle O\bmx, d\bmn-n\bmmu\rangle}\rd \bmx\\
&=\left(1+\OO\left(\frac{p^2(\ln n)^{3/2}}{n^{1/2}}\right)\right)\frac{p^{3/2}}{(2\pi)^{p-1}}\int_{\{\bmx\in \bR^{p-1}:\|\bmx\|^2_2\leq \fc p^2\ln n/n\}}e^{-\frac{dn}{2p}\|\bmx\|^2_2}e^{-\ri \langle \bmx, O^t(d\bmn-n\bmmu)\rangle}\rd \bmx\\
&=\left(1+\OO\left(\frac{p^2(\ln n)^{3/2}}{n^{1/2}}\right)\right)\frac{p^{3/2}}{(2\pi)^{p-1}}\int_{ \bR^{p-1}}e^{-\frac{dn}{2p}\|\bmx\|^2_2}e^{-\ri \langle \bmx, O^t(d\bmn-n\bmmu)\rangle}\rd \bmx+e^{-\left(\frac{\fc d}{2}+\oo(1)\right)p\ln n}\\
&=\left(1+\OO\left(\frac{p^2(\ln n)^{3/2}}{n^{1/2}}\right)\right)\frac{p^{3/2}}{(2\pi)^{p-1}}\int_{ \bR^{p-1}}e^{-\frac{dn}{2p}\|\bmx\|^2_2}e^{-\ri \langle \bmx, O^t(d\bmn-n\bmmu)\rangle}\rd \bmx+e^{-\left(\frac{\fc d}{2}+\oo(1)\right)p\ln n}\\
&=\left(1+\OO\left(\frac{p^2(\ln n)^{3/2}}{n^{1/2}}\right)\right)p^{3/2}
\left(\frac{p}{2\pi d n}\right)^{\frac{p-1}{2}}e^{-\frac{dpn}{2}\left\|O^t \left(\frac{\bmn}{n}-\frac{\bmmu}{d}\right)\right\|_2^2}
+e^{-\left(\frac{\fc d}{2}+\oo(1)\right)p\ln n}.
\end{split}\end{align}
The exponents in \eqref{e:firstfactord} and \eqref{e:integ4d} cancel
\begin{align*}
-\frac{dpn}{2}\left\|O^t \left(\frac{\bmn}{n}-\frac{\bmmu}{d}\right)\right\|_2^2+\frac{(d-1)pn}{2}\sum_{j=0}^{p-1}(\fn_j-1/p)^2
=-\frac{pn}{2}\left\|O^t \left(\frac{\bmn}{n}-\frac{\bmmu}{d}\right)\right\|_2^2,
\end{align*}
where we used that $O^t$ is an isometry from $\{(x_1,x_2,\cdots, x_{p})\in \bR^p: x_1+x_2+\cdots+x_p=0\}$ to $\bR^{p-1}$.
Therefore, by combining the estimates \eqref{e:firstfactord}, \eqref{e:integ1d}, \eqref{e:integ2d}, \eqref{e:integ3d} and \eqref{e:integ4d}, we conclude that for any $p$-tuple $(n_0,n_1,\cdots,n_{p-1})\in \cE$, with $\sum_{j=0}^{p-1}jn_j\equiv 0 \Mod p$,
\begin{align}\begin{split}\label{e:integ4.5}
&\phantom{{}={}}\frac{1}{|\mathsf{M}_{d,p}|}\sum_{\bmv\in \cS(n_0,n_1,\cdots, n_{p-1})} |\{\cG\in \mathsf{M}_{n,d}: A(\cG)\bmv=\bm0\}|\\
&=
\left(1+\OO\left(\frac{p^2(\ln n)^{3/2}}{n^{1/2}}\right)\right)
p^{3/2}
\left(\frac{p}{2\pi n}\right)^{\frac{p-1}{2}}e^{-\frac{pn}{2}\left\|O^t \left(\frac{\bmn}{n}-\frac{\bmmu}{d}\right)\right\|_2^2}+e^{-\left(\frac{\fc d}{2}-\frac{(d-1)\fb }{2}+\oo(1)\right)p\ln n}.
\end{split}\end{align}
For the second term on the righthand side of \eqref{e:integ4.5}, since the total number of $p$-tuples $(n_0,n_1,\cdots, n_{p-1})\in \cE$ is bounded by $e^{p\ln n}$,
\begin{align}\label{e:secondtermd}
\sum_{(n_0,n_1,\cdots,n_{p-1}) \in \cE}e^{-\left(\frac{\fc d}{2}-\frac{(d-1)\fb }{2}+\oo(1)\right)p\ln n}
=e^{-\left(\frac{\fc d}{2}-\frac{(d-1)\fb}{2}-1+\oo(1)\right)p\ln n},
\end{align}
which is negligible provided $\fc$ is large enough.

For the first term on the righthand side of \eqref{e:integ4.5} corresponding to the $p$-tuple $(n_0,n_1,\cdots,n_{p-1})\in \cE$, with $\sum_{j=0}^{p-1}jn_j\equiv 0 \Mod p$, we can replace it by an average.
\begin{align*}\begin{split}
pe^{-\frac{pn}{2}\left\|O^t \left(\frac{\bmn}{n}-\frac{\bmmu}{d}\right)\right\|_2^2}
&=e^{-\frac{pn}{2}\left\|O^t \left(\frac{\bmn}{n}-\frac{\bmmu}{d}\right)\right\|_2^2}
+\left(1+\OO\left(\frac{p(\ln n)^{1/2}}{n^{1/2}}\right)\right)
\sum_{j=1}^{p-1}e^{-\frac{pn}{2}\left\|O^t \left(\frac{\bmn+\bme_j-\bme_0}{n}-\frac{\bmmu}{d}\right)\right\|_2^2}.
\end{split}\end{align*}
Therefore, we can replace the sum over $p$-tuples $(n_0,n_1,\cdots,n_{p-1})\in \cE$, with $\sum_{j=0}^{p-1}jn_j\equiv 0 \Mod p$ to the sum over all $p$-tuples $(n_0,n_1,\cdots,n_{p-1})\in \cE$ with a factor $1/p$. 
\begin{align}\begin{split}\label{e:integ5d}
&\phantom{{}={}}\sum_{(n_0,n_1,\cdots,n_{p-1})\in \cE\atop \sum_{j=0}^{p-1}jn_j\equiv 0\Mod p}
\left(1+\OO\left(\frac{p^2(\ln n)^{3/2}}{n^{1/2}}\right)\right)
p^{3/2}
\left(\frac{p}{2\pi n}\right)^{\frac{p-1}{2}}e^{-\frac{pn}{2}\left\|O^t \left(\frac{\bmn}{n}-\frac{\bmmu}{d}\right)\right\|_2^2}\\
&=\sum_{(n_0,n_1,\cdots,n_{p-1})\in \cE}
\left(1+\OO\left(\frac{p^2(\ln n)^{3/2}}{n^{1/2}}\right)\right)
p^{1/2}
\left(\frac{p}{2\pi n}\right)^{\frac{p-1}{2}}e^{-\frac{pn}{2}\left\|O^t \left(\frac{\bmn}{n}-\frac{\bmmu}{d}\right)\right\|_2^2}.
\end{split}\end{align}
In the following we estimate the sum in \eqref{e:integ5d}. The set of points $O^t(\bmn/n-\bm\mu/d)$ for $\bmn=(n_0,n_1,\cdots, n_{p-1})\in \cE$ is a subset of a lattice in $\bR^{p-1}$. A set of base for this lattice is given by
\begin{align*}
O^t(e_{j}-e_{0})/n, \quad 0\leq j\leq p-1.
\end{align*}
The volume of the fundamental domain is $p^{1/2}/n^{p-1}$. By viewing \eqref{e:integ5d} as a Riemann sum, we can rewrite it as an integral on $\bR^{p-1}$.
\begin{align}\begin{split}\label{e:integ6d}
&\phantom{{}={}}\sum_{(n_0,n_1,\cdots,n_{p-1})\in \cE}
\left(1+\OO\left(\frac{p^2(\ln n)^{3/2}}{n^{1/2}}\right)\right)
p^{1/2}
\left(\frac{p}{2\pi n}\right)^{\frac{p-1}{2}}e^{-\frac{pn}{2}\left\|O^t \left(\frac{\bmn}{n}-\frac{\bmmu}{d}\right)\right\|_2^2}\\
&=\left(1+\OO\left(\frac{p^2(\ln n)^{3/2}}{n^{1/2}}\right)\right)
\left(\frac{pn}{2\pi}\right)^{\frac{p-1}{2}}\int_{\{\bmx\in\bR^{p-1}:\|\bmx\|_2^2\leq \fb\ln n/n\}} e^{-\frac{pn}{2}\|\bmx\|^2}\rd \bmx\\
&=\left(1+\OO\left(\frac{p^2(\ln n)^{3/2}}{n^{1/2}}\right)\right),
\end{split}\end{align}
provided $\fb$ is large enough. The claim \eqref{e:equd} follows from combining \eqref{e:integ4.5}, \eqref{e:secondtermd} and \eqref{e:integ6d}.
%
%
This finishes the proof of Proposition \ref{p:lcltd}.
\end{proof}

\subsection{Large deviation estimate}
\label{subs:ldpd}

In this section, we show that the sum of terms in \eqref{e:keyestd} corresponding to non-equidistributed $p$-tuples $(n_0,n_1,\cdots,n_{p-1})$ is small. 
\begin{proposition}\label{p:ldpd}
Let $d\geq 3$ be a fixed integer, and a prime number $p$ such that $\gcd(p,d)=1$ and $p\ll n^{(d-2)/2d}$. Then for $n$ sufficiently large,
\begin{align}\label{e:noneqd}
\frac{1}{|\mathsf{M}_{n,d}|}\sum_{(n_0,n_1,\cdots,n_{p-1}) \in \cN}\sum_{\bmv\in \cS(n_0,n_1,\cdots, n_{p-1})} |\{\cG\in \mathsf{M}_{n,d}: A(\cG)\bmv=\bm0\}|\leq \frac{\OO(p^2)}{n^{(d-2)}}.
\end{align}
\end{proposition}
Thanks to Proposition \ref{p:walkrepd}, we have
\begin{align}\begin{split}\label{e:expldpd}
&\phantom{{}={}}\frac{1}{|\mathsf{M}_{n,d}|}\sum_{(n_0,n_1,\cdots,n_{p-1}) \in \cN}\sum_{\bmv\in \cS(n_0,n_1,\cdots, n_{p-1})} |\{\cG\in \mathsf{M}_{n,d}: A(\cG)\bmv=\bm0\}|\\
&=\sum_{(n_0,n_1,\cdots,n_{p-1})\in \cN}{n\choose n_0,n_1,\cdots, n_{p-1}}{dn\choose dn_0,dn_1,\cdots, dn_{p-1}}^{-1}\\
&\phantom{{}={}}\times p^{n(d-1)}\bP(X_1+X_2+\cdots+X_n=(dn_0, dn_{1},\cdots, dn_{p-1})),
\end{split}\end{align}
where $X_1,X_2,\cdots, X_n$ are independent copies of $X$, which is uniform distributed over $\cU_{d,p}$ as defined in \eqref{e:defcU}. We enumerate the elements of $\cU_{d,p}$ as
\begin{align*}
\cU_{d,p}=\{\bmw_1, \bmw_2,\cdots, \bmw_{p^{d-1}}\}, \quad \bmw_1=(d,0,0,\cdots,0).
\end{align*}
For any non-equidistributed $p$-tuple $(n_0,n_1,\cdots,n_{p-1})$, we denote $\fn_j=n_j/n$ for $j=0,1,\cdots, p-1$. We estimate the first factor on the righthand side of \eqref{e:expldpd} using Stirling's formula,
\begin{align}\begin{split}\label{e:firstfactor2d}
{n\choose n_0,n_1,\cdots, n_{p-1}}{dn\choose dn_0,dn_1,\cdots, dn_{p-1}}^{-1}p^{n(d-1)}
\leq e^{\OO(p)}\exp\left\{(d-1)n\left(\ln p+\sum_{j=0}^{p-1} \fn_j\ln \fn_j\right)\right\}.
\end{split}\end{align}
For the random walk term in \eqref{e:expldpd}, we have the following large deviation bound
\begin{align}\label{e:secondfactor2d}
\bP(X_1+X_2+\cdots+X_n=(dn_0, dn_{1},\cdots, dn_{p-1}))
\leq \exp\left\{n\inf_{\bmt\in \bR^p} \log \bE[e^{\langle \bmt, X\rangle}]-d\langle \bmt, \bm\fn\rangle\right\}.
\end{align}
Thus combining \eqref{e:firstfactor2d} and \eqref{e:secondfactor2d}, we get that
\begin{align*}
\begin{split}
&\phantom{{}={}}\frac{1}{|\mathsf{M}_{n,d}|}\sum_{(n_0,n_1,\cdots,n_{p-1}) \in \cN}\sum_{\bmv\in \cS(n_0,n_1,\cdots, n_{p-1})} |\{\cG\in \mathsf{M}_{n,d}: A(\cG)\bmv=\bm0\}|\\
&\leq \sum_{(n_0,n_1,\cdots,n_{p-1}) \in \cN} e^{\OO(p)}\exp\left\{(d-1)n \ln p+(d-1)n\sum_{j=0}^{p-1} \fn_j\ln \fn_j+n\inf_{\bmt\in \bR^p} \log \bE[e^{\langle \bmt, X\rangle}]-d\langle \bmt, \bm\fn\rangle\right\}\\&=\sum_{(n_0,n_1,\cdots,n_{p-1}) \in \cN} e^{\OO(p)} e^{n I(\fn_0,\fn_1, \cdots, \fn_{p-1})},
\end{split}\end{align*}
where the rate function is given by
\begin{align}\label{e:ratefd}
I(\fn_0,\fn_1, \cdots, \fn_{p-1})=(d-1)\ln p+(d-1)\sum_{j=0}^{p-1} \fn_j\ln \fn_j+\inf_{\bmt\in \bR^p} \log \bE[e^{\langle \bmt, X\rangle}]-d\langle \bmt, \bm\fn\rangle.
\end{align}
The rate function function is negative except for two points:  $\fn_0=\fn_1=\cdots=\fn_{p-1}=1/p$,
and  $\fn_0=1$, $\fn_1=\cdots=\fn_{p-1}=0$. In the following proposition we give a quantitative estimate of the rate function.
\begin{proposition}\label{p:ldpbound}
Let $d\geq 3$ be a fixed integer, and a prime number $p$ such that $\gcd(p,d)=1$. The rate function as defined in \eqref{e:ratefd} satisfies: for any small $\delta> 0$, there exists a constant $c(\delta)>0$, such that
\begin{align}\label{e:ratebound}
I(\fn_0, \fn_1,\cdots, \fn_{p-1})\leq -\frac{c(\delta)}{p},
\end{align}
unless
$ \max_{0\leq k\leq p}|\fn_k-1/p|\leq \delta/p$, or
$\fn_0\geq 1-\delta/p$.
\end{proposition}

\begin{proof}
We take $\bmt=(d-1)/d((\ln \fn_0, \ln \fn_1,\cdots, \ln \fn_{p-1})+\ln p)$ in \eqref{e:ratefd}, the rate function $I$ is upper bounded by 
\begin{align}\label{e:sharprate}
I(\fn_0,\fn_1, \cdots, \fn_{p-1})\leq \log \sum_{j=1}^{p^{d-1}}\prod_{k=0}^{p-1}\fn_k^{\frac{d-1}{d}\bmw_j(k)},
\end{align}
In the following, we prove that there exists a constant $c(\delta)>0$
\begin{align}\label{e:ratebound2}
\sum_{j=1}^{p^{d-1}}\prod_{k=0}^{p-1}\fn_k^{\frac{d-1}{d}\bmw_j(k)}\leq 1-\frac{c(\delta)}{p}, 
\end{align}
unless $ \max_{0\leq k\leq p}|\fn_k-1/p|\leq \delta/p$, or
$\fn_0\geq 1-\delta/p$. Then the claim \eqref{e:ratebound} follows.

For any $\varepsilon>0$ and $d$-tuple $(a_1,a_2,\cdots,a_{d})\in \bF_p^d$ such that $a_1+a_2+\cdots+a_d=0$, if
\begin{align*}
\frac{\min_{1\leq r\leq d}\fn_{a_r}}{\max_{1\leq r\leq d}\fn_{a_r}}\leq 
\frac{1}{1+\varepsilon}, 
\end{align*}
then there exists a constant $c(\varepsilon)>0$ such that 
\begin{align*}
\prod_{r=1}^d\fn_{a_r}^{\frac{d-1}{d}}\leq \frac{1-c(\varepsilon)}{d}\sum_{r=1}^d \prod_{1\leq s\leq d\atop s\neq r}\fn_{a_s}.
\end{align*}
Therefore, by the defining relation of the multiset $\cU_{d,p}$ as in \eqref{e:defcU}, we have
\begin{align}\begin{split}\label{e:am-gm}
\sum_{j=1}^{p^{d-1}}\prod_{k=0}^{p-1}\fn_k^{\frac{d-1}{d}\bmw_j(k)}
&=\sum_{(a_1,a_2,\cdots,a_{d})\in \bF_p^d,\atop a_1+a_2+\cdots+a_d=\bm0}\prod_{r=1}^d \fn_{a_r}^{\frac{d-1}{d}}\\
&\leq \sum_{(a_1,a_2,\cdots,a_{d})\in \bF_p^d,\atop a_1+a_2+\cdots+a_d=\bm0}\frac{1}{d}\left({1-1_{\frac{\min_{r} \fn_{a_r}}{\max_{r} \fn_{a_r}}\leq \frac{1}{1+\varepsilon}}c(\varepsilon)}\right)\sum_{r=1}^d\prod_{1\leq s\leq d\atop s\neq r} \fn_{a_s}\\
&=\left(\sum_{j=0}^{p-1} \fn_j\right)^{d-1}-\frac{c(\varepsilon)}{d} 
\sum_{(a_1,a_2,\cdots,a_{d})\in \bF_p^d,\atop a_1+a_2+\cdots+a_d=\bm0}1_{\frac{\min_{r} \fn_{a_r}}{\max_{r} \fn_{a_r}}\leq \frac{1}{1+\varepsilon}}\sum_{r=1}^d\prod_{1\leq s\leq d\atop s\neq r} \fn_{a_s}\\
&\leq 1-\frac{c(\varepsilon)}{d} 
\sum_{(a_1,a_2,\cdots,a_{d})\in \bF_p^d,\atop a_1+a_2+\cdots+a_d=\bm0}1_{\frac{\min_{r} \fn_{a_r}}{\max_{r} \fn_{a_r}}\leq \frac{1}{1+\varepsilon}}\prod_{r=1}^{d-1}\fn_{a_r}.
\end{split}\end{align}
In the following, we take $\varepsilon=\delta/3$, and prove that there exists a constant $c(\delta)$
\begin{align}\label{e:ratebound3}
\sum_{(a_1,a_2,\cdots,a_{d})\in \bF_p^d,\atop a_1+a_2+\cdots+a_d=\bm0}1_{\frac{\min_{r} \fn_{a_r}}{\max_{r} \fn_{a_r}}\leq \frac{1}{1+\varepsilon}}\prod_{r=1}^{d-1}\fn_{a_r}\geq \frac{c(\delta)}{p}, 
\end{align}
unless $ \max_{0\leq k\leq p}|\fn_k-1/p|\leq \delta/p$, or
$\fn_0\geq 1-\delta/p$. Then the claim \eqref{e:ratebound2} follows.

We sort these numbers $\fn_0, \fn_1,\cdots, \fn_{p-1}$ as $\fn_{k_1}\geq \fn_{k_2}\geq \cdots \geq \fn_{k_{p}}$, then $\fn_{k_1}\geq 1/p$. We take the indices $t_1$ and $t_2$ such that $\fn_{k_{t_1}}>\fn_{k_1}/(1+\varepsilon)\geq \fn_{k_{t_1+1}}$, and $\fn_{k_{t_2}}>\fn_{k_1}/(1+\varepsilon)^2\geq \fn_{k_{t_2+1}}$. If $\fn_{k_{t_1+1}}+\fn_{k_{t_1+2}}+\cdots +\fn_{k_p}\geq \varepsilon$, by restricting the sum in \eqref{e:ratebound3} over $d$-tuples with $a_1=k_1$ and $a_2\in \{k_{t_1+1}, k_{t_1+2},\cdots, k_{p}\}$, we get
\begin{align*}
\sum_{(a_1,a_2,\cdots,a_{d})\in \bF_p^d,\atop a_1+a_2+\cdots+a_d=\bm0}1_{\frac{\min_{r} \fn_{a_r}}{\max_{r} \fn_{a_r}}\leq \frac{1}{1+\varepsilon}}\sum_{r=1}^{d-1}\fn_{a_r}
\geq 
\fn_{k_1}(\fn_{k_{t_1+1}}+\fn_{k_{t_1+2}}+\cdots+\fn_{k_{p}})\sum_{a_3, a_4,\cdots, a_{d-1}\in \bF_p}\prod_{r=3}^{d-1}\fn_{a_r}\geq \frac{\varepsilon}{p}.
\end{align*}
The claim \eqref{e:ratebound3} follows. So we can assume that $\fn_{k_{t_1+1}}+\fn_{k_{t_1+2}}+\cdots +\fn_{k_p}\leq  \varepsilon$ and thus $\fn_{k_{1}}+\fn_{k_{2}}+\cdots +\fn_{k_{t_1}}\geq  1-\varepsilon$.
If $\fn_{k_{t_2+1}}+\fn_{k_{t_2+2}}+\cdots +\fn_{k_p}\geq \varepsilon/p$, by restricting the sum in \eqref{e:ratebound3} over $d$-tuples with $a_1\in \{k_1, k_2,\cdots, k_{t_1}\}$ and $a_2\in \{k_{t_2+1}, k_{t_2+2},\cdots, k_{p}\}$, we get
\begin{align*}\begin{split}
&\phantom{{}={}}\sum_{(a_1,a_2,\cdots,a_{d})\in \bF_p^d,\atop a_1+a_2+\cdots+a_d=\bm0}1_{\frac{\min_{r} \fn_{a_r}}{\max_{r} \fn_{a_r}}\leq \frac{1}{1+\varepsilon}}\sum_{r=1}^{d-1}\fn_{a_r}\\
&\geq 
(\fn_{k_{1}}+\fn_{k_{2}}+\cdots+\fn_{k_{t_1}})(\fn_{k_{t_2+1}}+\fn_{k_{t_2+2}}+\cdots+\fn_{k_{p}})\sum_{a_3, a_4,\cdots, a_{d-1}\in \bF_p}\prod_{r=3}^{d-1}\fn_{a_r}\geq \frac{(1-\varepsilon)\varepsilon}{p}.
\end{split}\end{align*}
The claim \eqref{e:ratebound3} follows. So we can assume that $\fn_{k_{t_2+1}}+\fn_{k_{t_2+2}}+\cdots +\fn_{k_p}\leq  \varepsilon/p$. There are three cases:
\begin{enumerate}
\item If $t_2=p$, then we have $\max_{0\leq k\leq p-1} \fn_k \leq (1+\varepsilon)^2\min_{0\leq k\leq p-1} \fn_k$. And it follows that $(1+\varepsilon)^{-2}/p\leq \fn_k \leq(1+\varepsilon)^{2}/p$ for $0\leq k\leq p-1$. The claim follows by taking $2\varepsilon+\varepsilon^2\leq \delta$.
\item If $t_2=1$, then we have $\fn_{k_1}\geq 1-\varepsilon/p$. If $k_1=0$, then the claim follows by taking $\varepsilon\leq \delta$. Otherwise, $k_1\neq 0$. By restricting the sum in \eqref{e:ratebound3} over $d$-tuples with $a_1=a_2=\cdots=a_{d-1}=k_1$, we get
\begin{align*}
\sum_{(a_1,a_2,\cdots,a_{d})\in \bF_p^d,\atop a_1+a_2+\cdots+a_d=\bm0}1_{\frac{\min_{r} \fn_{a_r}}{\max_{r} \fn_{a_r}}\leq \frac{1}{1+\varepsilon}}\sum_{r=1}^{d-1}\fn_{a_r}
\geq 
\fn_{k_1}^{d-1}\geq \left(1-\frac{\varepsilon}{p}\right)^{d-1}.
\end{align*}
The claim \eqref{e:ratebound3} follows. 
\item If $1<t_2<p$, then we have $\fn_{k_1}\geq \fn_{k_2}\geq \cdots\geq \fn_{k_{t_2}}\geq \fn_{k_1}/(1+\varepsilon)^2$. And it follows $(1+\varepsilon)^{2}/t_2\geq \fn_{k_1}\geq \fn_{k_2}\geq \cdots\geq \fn_{k_{t_2}} \geq(1+\varepsilon)^{-2}/t_2\geq (1+\varepsilon)\fn_{k_{t_2}+1}$. We will restrict the sum in \eqref{e:ratebound3} over $d$-tuples with $a_1,a_2,\cdots, a_{d-1}\in \{k_1, k_2,\cdots, k_{t_2}\}$, and $a_d\in \{k_{t_2+1}, k_{t_2+2}, \cdots, k_p\}$.
We take integer $q$ such that $t_2q\equiv -2(k_1+k_2+\cdots+k_{t_2})\Mod p$. The number of $(d-2)$-tuples $a_1,a_2,\cdots, a_{d-2}\in \{k_1, k_2,\cdots, k_{t_2}\}$ such that $a_1+a_2+\cdots+a_{d-2}\not\equiv q\Mod p$
is at least $(t_2-1)t_2^{d-3}$. For any $(d-2)$-tuple $a_1,a_2,\cdots, a_{d-2}\in \{k_1, k_2,\cdots, k_{t_2}\}$ such that $a_1+a_2+\cdots+a_{d-2}\not\equiv q\Mod p$, there exists at least one $a_{d-1}\in \{k_1, k_2,\cdots, k_{t_2}\}$ such that $a_1+a_2+\cdots+a_{d-1}\not\equiv -k_1,-k_2,\cdots,-k_{t_2}\Mod p$. Otherwise 
\begin{align*}
t_2(a_1+a_2+\cdots+a_{d-2})+(k_1+k_2+\cdots+k_{t_2})\equiv -(k_1+k_2+\cdots+k_{t_2})\Mod p,
\end{align*}
and $a_1+a_2+\cdots+a_{d-2}\equiv q\Mod p$. This leads to a contradiction. Therefore, the number of $d$-tuples with $a_1,a_2,\cdots, a_{d-1}\in \{k_1, k_2,\cdots, k_{t_2}\}$, and $a_d\in \{k_{t_2+1}, k_{t_2+2}, \cdots, k_p\}$ is at least $(t_2-1)t_2^{d-3}$. By restricting the sum in \eqref{e:ratebound3} over $d$-tuples with $a_1,a_2,\cdots, a_{d-1}\in \{k_1, k_2,\cdots, k_{t_2}\}$, and $a_d\in \{k_{t_2+1}, k_{t_2+2}, \cdots, k_p\}$, we get
\begin{align*}
\sum_{(a_1,a_2,\cdots,a_{d})\in \bF_p^d,\atop a_1+a_2+\cdots+a_d=\bm0}1_{\frac{\min_{r} \fn_{a_r}}{\max_{r} \fn_{a_r}}\leq \frac{1}{1+\varepsilon}}\sum_{r=1}^{d-1}\fn_{a_r}
\geq 
(t_2-1)t_2^{d-3} \left(\frac{1}{(1+\varepsilon)^2 t_2}\right)^{d-1}
\geq \frac{1}{2(1+\varepsilon)^{2d-2}p} .
\end{align*}
The claim \eqref{e:ratebound3} follows. 
\end{enumerate}
This finishes the proof of Proposition \ref{p:ldpbound}. 
\end{proof}

\begin{proof}[Proof of Proposition \ref{p:ldp}]
We further decompose the set of non-equidistributed $p$-tuples $(n_0,n_1,\cdots, n_{p-1})$ into four classes:
\begin{enumerate}
\item $p$-tuples $(n_0,n_1,\cdots, n_{p-1})\in \cN$ with $\max_{0\leq j\leq p-1}|n_j/n-1/p|\leq \delta/p$.
\item $p$-tuples $(n_0,n_1,\cdots, n_{p-1})\in \cN$  with $\fb\ln n/n<|n_0/n-1|\leq \delta/p$.
\item $p$-tuples $(n_0,n_1,\cdots, n_{p-1})\in \cN$  with $|n_0/n-1|\leq \fb\ln n/n$.
\item The remaining non-equidistributed $p$-tuples.
\end{enumerate}

For the first class, $\max_{0\leq j\leq p-1}|n_j/n-1/p|\leq \delta/p$. The total number of such $p$-tuples is $e^{\OO(p\ln n)}$. 
Given a $p$-tuple $(n_0,n_1,\cdots,n_{p-1})$ in the first class,
we will derive a more precise estimate of \eqref{e:ratebound2}, by a perturbation argument. Let 
\begin{align*}\begin{split}
&\fn_j=(1+\delta_j)/p,\quad j=0,1,\cdots, p-1.
\end{split}\end{align*}
where $\sum_{0\leq j\leq p-1}\delta_j=0$, $\max_{0\leq j\leq p-1}|\delta_j|\leq \delta$, and $\sum_j \delta_j^2\geq \fb p^2\ln n/n$.
We denote $\bm\delta=(\delta_0, \delta_1,\cdots, \delta_{p-1})$.
We use Taylor expansion, and rewrite the righthand side of \eqref{e:sharprate} as
\begin{align}\begin{split}\label{e:sumexp}
&\phantom{{}={}}\sum_{j=1}^{p^{d-1}}\prod_{k=0}^{p-1}\fn_k^{\frac{d-1}{d}\bmw_j(k)}
=\frac{1}{p^{d-1}}\sum_{j=1}^{p^{d-1}}\prod_{k=0}^{p-1}(1+\delta_k)^{\frac{d-1}{d}\bmw_j(k)}
=\frac{1}{p^{d-1}}\sum_{j=1}^{p^{d-1}}e^{\frac{d-1}{d}\sum_{k=0}^{p-1}\left(\delta_k-(1+\OO(\delta))\frac{\delta_k^2}{2}\right)\bmw_j(k)}\\
&=\frac{1}{p^{d-1}}\sum_{j=1}^{p^{d-1}}1+\frac{d-1}{d}\sum_{k=0}^{p-1}\left(\delta_k-(1+\OO(\delta))\frac{\delta_k^2}{2}\right)\bmw_j(k)
+(1+\OO(\delta))\frac{(d-1)^2}{2d^2}
\left(\sum_{k=0}^{p-1}\delta_k \bmw_j(k)\right)^2
\\
&=1+\frac{d-1}{2d p^{d-1}}\left(-(1+\OO(\delta))dp^{d-2}\|\bm\delta\|_2^2+(1+\OO(\delta))\frac{d-1}{d}\sum_{j=1}^{p^{d-1}}\langle \bm\delta, \bmw_j\rangle^2\right)
\end{split}\end{align}
We recall from \eqref{e:cov} that 
\begin{align*}
\sum_{j=1}^{p^{d-1}}\bmw_j\bmw_j^t=dp^{d-2}I_{p}+d(d-1)p^{d-3}\bm1\bm1^t.
\end{align*}
Therefore, 
\begin{align}\begin{split}\label{e:finalbound}
\sum_{j=1}^{p^{d-1}}\prod_{k=0}^{p-1}\fn_k^{\frac{d-1}{d}\bmw_j(k)}
&=1+\frac{d-1}{2d p^{d-1}}\left(-(1+\OO(\delta))dp^{d-2}\|\bm\delta\|_2^2+(1+\OO(\delta))(d-1)p^{d-2}\|\bm\delta\|_2^2\right)\\
&=1-(1+\OO(\delta))\frac{d-1}{2dp}\|\bm\delta\|_2^2\leq 1-(1+\OO(\delta))\frac{d-1}{2d}\frac{\fb p\ln n}{n}.
\end{split}\end{align}
Thanks to  \eqref{e:sharprate}, we obtain an upper bound for the rate function from \eqref{e:finalbound}
\begin{align*}
I(\fn_0,\fn_1,\cdots, \fn_{p-1})\leq -\left(\frac{(d-1)\fb}{2d}+\oo(1)\right)\frac{p\ln n}{n}.
\end{align*}
The total contribution of terms in \eqref{e:noneqd} satisfying $\max_{0\leq j\leq p-1}|n_j/n-1/p|\leq \delta$ is bounded by
\begin{align}\label{e:case1d}
\exp\left\{-\left(\frac{(d-1)\fb}{2d}+\oo(1)\right)p\ln n+\OO(p\ln n)\right\}=\frac{\oo(1)}{n^{(d-2)}},
\end{align}
provided that we take $\fb$ sufficiently large.

For the second class, $\fb p\ln n/n<|n_0/n-1|\leq \delta/p$. The total number of such $p$-tuples is $e^{\OO(p\ln n)}$. Given a $p$-tuple $(n_0,n_1,\cdots,n_{p-1})$ in the second class, we will derive a more precise estimate of \eqref{e:ratebound2}, by a perturbation argument. Let $\fn_0=1-\delta_0$, where $\delta_0\leq \delta/p$.

We decompose the $d$-tuples $(a_1,a_2,\cdots,a_{d})\in \bF_p^d$ such that $a_1+a_2+\cdots+a_d=0$ into three sets: $\{(0,0,\cdots,0)\}$ and 
\begin{align*}\begin{split}
&\cA_1=\{(a_1,a_2,\cdots,a_{d})\in \bF_p^d: a_1+a_2+\cdots+a_d=0, \sum_{r=1}^d \bm1(a_r=0)=0\},\\
&\cA_2=\{(a_1,a_2,\cdots,a_{d})\in \bF_p^d: a_1+a_2+\cdots+a_d=0, 0<\sum_{r=1}^d \bm1(a_r=0)<d\}.
\end{split}\end{align*}
For any $d$-tuple $(a_1, a_2,\cdots, a_d)\in \cA_2$, we have $\min_{1\leq r\leq d}\fn_{a_r}/ \max_{1\leq r\leq d}\fn_{a_r}\leq \delta_0/(1-\delta_0)$, and 
\begin{align*}
\prod_{r=1}^d\fn_{a_r}^{\frac{d-1}{d}}\leq \left(4\delta_0\right)^{1/d}\frac{1}{d}\sum_{r=1}^d \prod_{1\leq s\leq d\atop s\neq r}\fn_{a_s}.
\end{align*}
Therefore, by the defining relation of the multiset $\cU_{d,p}$ as in \eqref{e:defcU}, we have
\begin{align}\begin{split}\label{e:midbound}
&\phantom{{}={}}\sum_{j=1}^{p^{d-1}}\prod_{k=0}^{p-1}\fn_k^{\frac{d-1}{d}\bmw_j(k)}
=\sum_{(a_1,a_2,\cdots,a_{d})\in \bF_p^d,\atop a_1+a_2+\cdots+a_d=\bm0}\prod_{r=1}^d \fn_{a_r}^{\frac{d-1}{d}}\\
&\leq\fn_0^{d-1}+\sum_{(a_1,a_2,\cdots,a_d)\in \cA_1}\frac{1}{d}\sum_{r=1}^d \prod_{1\leq s\leq d\atop s\neq r}\fn_{a_s}+\left(4\delta_0\right)^{1/d}\sum_{(a_1,a_2,\cdots,a_d)\in \cA_2}\frac{1}{d}\sum_{r=1}^d \prod_{1\leq s\leq d\atop s\neq r}\fn_{a_s}\\
&\leq\fn_0^{d-1}+(\fn_1+\fn_2+\cdots+\fn_{p-1})^{d-1}+\left(4\delta_0\right)^{1/d}(1-\fn_0^{d-1})\\
&=1+\delta_0^{d-1}-\left(1-\left(4\delta_0\right)^{1/d}\right)\left(1-\left(1-\delta_0\right)^{d-1}\right).
\end{split}\end{align}
Therefore, thanks to \eqref{e:sharprate}, we get
\begin{align*}
I(\fn_0,\fn_1,\cdots, \fn_{p-1})\leq -(1+\oo(1))(d-1)\delta_0
\end{align*} 
provided that $\delta_0$ is small enough. Thus, the total contribution of terms in \eqref{e:noneqd} satisfying $\fb p\ln n/n<|n_0/n-1|\leq \delta/p$ is bounded by
\begin{align}\label{e:case2d}
\exp\left\{-(1+\oo(1))\fb (d-1)p\ln n+\OO(p\ln n)\right\}=\frac{\oo(1)}{n^{(d-2)}},
\end{align}
provided that we take $\fb$ sufficiently large.

For the third class, $|n_0/n-1|\leq \fb p\ln n/n$. 
We rewrite \eqref{e:expldpd} in terms of number of walk paths,
\begin{align}\begin{split}\label{e:expldpd2}
&\sum_{(n_0,n_1,\cdots,n_{p-1})\in \cN}{n\choose n_0,n_1,\cdots, n_{p-1}}{dn\choose dn_0,dn_1,\cdots, dn_{p-1}}^{-1}\\
&\times|\{(\bmu_1,\bmu_2\cdots, \bmu_n)\in \cU_{d,p}^n: \bmu_1+\bmu_2+\cdots+\bmu_n=(dn_0, dn_{1},\cdots, dn_{p-1})\}|.
\end{split}\end{align}
Given a $p$-tuple $(n_0,n_1,\cdots,n_{p-1})$, in the third class with $n_0=n-m$ and $2\leq m\leq \fb p\ln n$, we reestimate the first factor on the righthand side of \eqref{e:expldpd2},
\begin{align}\begin{split}\label{e:entr}
{n\choose n_0,n_1,\cdots, n_{p-1}}{dn\choose dn_0,dn_1,\cdots, dn_{p-1}}^{-1}
\leq \frac{e^{\OO(m)}}{n^{(d-1)m}}\frac{(dn_1)! (dn_2)!\cdots (dn_{p-1})!}{n_1!n_2!\cdots n_{p-1}!}.
\end{split}\end{align}
For the number of walk paths in \eqref{e:expldpd2}, we recall that
\begin{align*}
\cU_{d,p}=\{\bmw_1, \bmw_2,\cdots, \bmw_{p^{d-1}}\}, \quad \bmw_1=(d,0,0,\cdots,0),
\end{align*}
and notice that $\bmw_j(1)+\bmw_j(2)+\cdots \bmw_j(p-1)\geq 2$ for $2\leq j\leq p^{d-1}$. Moreover, since $n_0=n-m$, we have $\bmu_1+\bmu_2+\cdots+\bmu_n=(dn_0, dn_{1},\cdots, dn_{p-1})$, with $dn_1+dn_2+\cdots+dn_{p-1}=dm$.
Therefore,  $\bmu_i=\bmw_1$ for all $1\leq i\leq n$, except for at most $dm/2$ of them, and we have
\begin{align}\begin{split}\label{e:pathest}
&\phantom{{}={}}|\{(\bmu_1,\bmu_2\cdots, \bmu_n)\in \cU_{d,p}^n: \bmu_1+\bmu_2+\cdots+\bmu_n=(dn_0, dn_{1},\cdots, dn_{p-1})\}|\\
&\leq \frac{(dm)!}{(dn_1)! (dn_2)!\cdots (dn_{p-1})!} \sum_{k=1}^{dm/2}n^k{dm-k-1\choose k-1}\leq\OO(1)\frac{(dm)!}{(dn_1)! (dn_2)!\cdots (dn_{p-1})!} n^{dm/2}.
\end{split}\end{align}
Putting \eqref{e:entr} and \eqref{e:pathest} together, the total contribution of terms in \eqref{e:noneqd} satisfying $|n_0/n-1|\leq \fb p\ln n/n$ is bounded by
\begin{align}\begin{split}\label{e:case3d}
&\phantom{{}={}}\sum_{m=2}^{\fb p\ln n}\sum_{n_1+n_2+\cdots+n_{p-1}=m}\frac{e^{\OO(m)}(dm)!}{n_1! n_2!\cdots n_{p-1}!}\frac{1}{n^{(d/2-1)m}}\\
&\leq
\sum_{m=2}^{\fb p\ln n}
\frac{e^{\OO(m)}(dm)!p^m}{m!n^{(d/2-1)m}}
\leq \sum_{m=2}^{\fb p\ln n}\left(\frac{\OO(1)p m^d}{m n^{d/2-1}}\right)^m
\leq \frac{\OO(1)p^2}{n^{d-2}}.
\end{split}\end{align}
provided that $p\ll n^{(d-2)/2d}$.

For the last class, the total number of such $p$-tuples is $e^{\OO(p\ln n)}$, and thanks to Proposition \ref{p:ldpbound}, each term is exponentially small, i.e. $e^{-c(\delta)n/p}$. Therefore the total contribution is
\begin{align}\label{e:case4d}
\exp\{-c(\delta) n/p+\OO(p\ln n)\}=\frac{\oo(1)}{n^{(d-2)}}.
\end{align}

The claim \eqref{e:noneqd} follows from combining the discussion of all four cases, \eqref{e:case1d}, \eqref{e:case2d}, \eqref{e:case3d} and \eqref{e:case4d}.
\end{proof}

\section{Proof of Theorem \ref{thm:smcount} for random undirected $d$-regular graphs}
Thanks to Proposition \ref{p:walkrepu}, we can rewrite the lefthand side of \eqref{e:smcountu} as
\begin{align*}
&\phantom{{}={}}\sum_{\bmv\in \bF_p^n\setminus{\bm0}}|\{\cG\in \mathsf{G}_{n,d}: A(\cG)\bmv=\bm0\}|
=\sum_{n_0,n_1,\cdots,n_{p-1}\in \bZ_{\geq 0}, n_0<n\atop n_0+n_1+\cdots+n_{p-1}=n}\sum_{\bmv\in \cS(n_0,n_1,\cdots, n_{p-1})} |\{\cG\in \mathsf{G}_{n,d}: A(\cG)\bmv=\bm0\}|\\
&=\sum_{n_0,n_1,\cdots,n_{p-1}\in \bZ_{\geq 0}, n_0<n\atop n_0+n_1+\cdots+n_{p-1}=n}{n\choose n_0,n_1,\cdots, n_{p-1}}\sum_{M\in \cM(n_0, n_1,\cdots, n_{p-1})}\prod_{0\leq i<j\leq p-1}m_{ij}!\prod_{i=0}^{p-1}\frac{m_{ii}!}{2^{m_{ii}/2}(m_{ii}/2)!}\\
& \times p^{(d-1)n}\prod_{i=0}^{p-1}\bP(X_1+X_2+\cdots+X_{n_i}=(m_{i0},m_{i1},\cdots, m_{ip-1}))\\
&=\sum_{M\in \cM}{n\choose n_0(M),n_1(M),\cdots, n_{p-1}(M)}\prod_{0\leq i<j\leq p-1}m_{ij}!\prod_{i=0}^{p-1}\frac{m_{ii}!}{2^{m_{ii}/2}(m_{ii}/2)!}\\
 &\times p^{(d-1)n}\prod_{i=0}^{p-1}\bP(X_1+X_2+\cdots+X_{n_i}=(m_{i0},m_{i1},\cdots, m_{ip-1})),
\end{align*}
where for any $M\in\cM$, $n_i(M)\deq\sum_{j=0}^{p-1}m_{ij}/d$ for $i=0,1,\cdots, p-1$.
Therefore Theorem \ref{thm:smcount} is equivalent to the following estiamte
\begin{align}\begin{split}\label{e:keyest}
&\frac{2^{nd/2}(nd/2)!}{(nd)!}\sum_{M\in \cM}{n\choose n_0(M),n_1(M),\cdots, n_{p-1}(M)}\prod_{0\leq i<j\leq p-1}m_{ij}!\prod_{i=0}^{p-1}\frac{m_{ii}!}{2^{m_{ii}/2}(m_{ii}/2)!}\\
& \times p^{(d-1)n}\prod_{i=0}^{p-1}\bP(X_1+X_2+\cdots+X_{n_i}=(m_{i0},m_{i1},\cdots, m_{ip-1}))=1+\oo(1).
\end{split}\end{align}

To prove \eqref{e:keyest}, we fix a large number $\fb>0$, and decompose those $p\times p$ symmetric matrices $M=[m_{ij}]_{0\leq i,j\leq p-1}\in \cM$ into two classes:
\begin{enumerate}
\item (Equidistributed) $\cE$ is the set of $p\times p$ symmetric matrices $M=[m_{ij}]_{0\leq i,j\leq p-1}\in \cM$, such that $\sum_{i,j=0}^{p-1}(m_{ij}/(dn)-1/p^2)^2\leq \fb\ln n/n$.
\item (Non-equidistributed) $\cN$ is the set of $p\times p$ symmetric matrices $M=[m_{ij}]_{0\leq i,j\leq p-1}\in \cM$ which are not equidistributed.
\end{enumerate}
In Section \ref{subs:lclt}, we estimate the sum of terms in \eqref{e:keyest} corresponding to equidistributed $p\times p$ symmetric matrices using a local central limit theorem. In Section \ref{subs:ldp}, we show that the sum of terms in \eqref{e:keyest} corresponding to non-equidistributed $p\times p$ symmtric matrices is small, via a large deviation estimate. Theorem \ref{thm:smcount} for random directed $d$-regular graphs follows from combining Proposition \ref{p:lclt} and Proposition \ref{p:ldp}.

\subsection{Local central limit theorem estimate}
\label{subs:lclt}

In this section, we estimate the sum of terms in \eqref{e:keyest} corresponding to equidistributed $p\times p$ symmetric matrices $M=[m_{ij}]_{0\leq i,j\leq p-1}\in \cM$ using a local central limit theorem.
\begin{proposition}\label{p:lclt}
Let $d\geq 3$ be a fixed integer, and a prime number $p$ such that $\gcd(p,d)=1$ and $p\ll n^{1/8}$. Then
for $n$ sufficiently large
\begin{align}\label{e:equ}
\sum_{M \in \cE}\sum_{\bmv\in \cS(n_0(M),n_1(M),\cdots, n_{p-1}(M))} |\{\cG\in \mathsf{G}_{n,d}: A(\cG)\bmv=\bm0\}|=\left(1+\bm1(p=2)+\OO\left(\frac{p^4(\ln n)^{3/2}}{\sqrt{n}}\right)\right)|\mathsf{G}_{n,d}|.
\end{align}
\end{proposition}

\begin{proof}
Thanks to Proposition \ref{p:walkrepu}, we have
\begin{align}\begin{split}\label{e:explclt}
&\phantom{{}={}}\frac{1}{|\mathsf{G}_{n,d}|}\sum_{M \in \cE}\sum_{\bmv\in \cS(n_0(M),n_1(M),\cdots, n_{p-1}(M))} |\{\cG\in \mathsf{G}_{n,d}: A(\cG)\bmv=\bm0\}|\\
&=\frac{2^{nd/2}(nd/2)!}{(nd)!}\sum_{M\in \cE}{n\choose n_0(M),n_1(M),\cdots, n_{p-1}(M)}\prod_{0\leq i<j\leq p-1}m_{ij}!\prod_{i=0}^{p-1}\frac{m_{ii}!}{2^{m_{ii}/2}(m_{ii}/2)!}\\
& \phantom{{}={}}\times p^{(d-1)n}\prod_{i=0}^{p-1}\bP(X_1+X_2+\cdots+X_{n_i}=(m_{i0},m_{i1},\cdots, m_{ip-1})),
\end{split}\end{align}
where $X_1,X_2,\cdots, X_{n_i}$ are independent copies of $X$, which is uniform distributed over $\cU_{d,p}$ as defined in \eqref{e:defcU}.
In the rest of the proof, we simply write $n_i(M)$ as $n_i$ for $0\leq i\leq p-1$. 

For an equidistributed $p\times p$ symmetric matrix $M=[m_{ij}]_{0\leq i,j\leq p-1}$, we denote $\fm_{ij}=m_{ij}/(dn)$ for $i,j=0,1,\cdots, p-1$, and $\fn_i=n_i/n$ for $i=0,1,\cdots, p-1$. Then we have $\fn_i=\sum_{j=0}^{p-1}\fm_{ij}$ for $i=0,1,\cdots, p-1$. Moreover, by our definition of equidistributed, 
$\sum_{i,j=0}^{p-1}(\fm_{ij}-1/p^2)^2\leq \fb\ln n/n$, and by the AM-GM inequality
$\sum_{j=0}^{p-1}(\fn_j-1/p)^2\leq \fb p\ln n/n$. We estimate the first factor on the righthand side of \eqref{e:explclt} using Stirling's formula,
\begin{align}\begin{split}\label{e:firstfactor}
&\phantom{{}={}}\frac{2^{nd/2}(nd/2)!}{(nd)!}\sum_{M\in \cE}{n\choose n_0,n_1,\cdots, n_{p-1}}\prod_{0\leq i<j\leq p-1}m_{ij}!\prod_{i=0}^{p-1}\frac{m_{ii}!}{2^{m_{ii}/2}(m_{ii}/2)!}
p^{(d-1)n}\\
&=\left(1+\OO\left(\frac{p^4}{n}\right)\right)
2^{p/2}\sqrt{\pi n}\frac{\prod_{0\leq i<j\leq p-1}\sqrt{2\pi m_{ij}}}{\prod_{i=0}^{p-1}\sqrt{2\pi n_i}}
e^{\frac{dn}{2}\sum_{i,j=0}^{p-1}{\fm_{ij}}\ln \fm_{ij}-n\sum_{j=0}^{p-1}\fn_j\ln \fn_j }p^{(d-1)n}\\
&=
\left(1+\OO\left(\frac{p^3(\ln n)^{3/2}}{\sqrt{n}}\right)\right)
2^{p/2}\sqrt{\pi n}\frac{\prod_{0\leq i<j\leq p-1}\sqrt{2\pi dn/p^2}}{\prod_{i=0}^{p-1}\sqrt{2\pi n/p}}
e^{\frac{dnp^2}{4}\sum_{i,j=0}^{p-1}\left(\fm_{ij}-\frac{1}{p^2}\right)^2-\frac{np}{2}\sum_{j=0}^{p-1}\left(\fn_j-\frac{1}{p}\right)^2}.
\end{split}\end{align}

In the following, we estimate the second factor on the righthand side of \eqref{e:explclt}, i.e. $\bP(S_{n_i}=(m_{i0},m_{i1},\cdots, m_{ip-1}))$, where $S_{n_i}=X_1+X_2+\cdots+X_{n_i}$.
We recall that $X_1, X_2,\cdots, X_{n_i}$ are independent copies of $X$, which is uniformly distributed over the multiset $\cU_{d,p}$ as defined in \eqref{e:defcU}. 
We use the notation $\bmm_i=(m_{i0}, m_{i1}, m_{i2}, \cdots, m_{ip-1})$. By inverse Fourier formula
\begin{align*}\begin{split}
\bP(S_{n_i}=\bmm_i)
&=\frac{1}{(2\pi)^p}\int_{2\pi \bR^p/\bZ^p}\phi_X^{n_i}(\bmt)e^{-\ri \langle \bmt, \bmm_i\rangle}\rd \bmt\\
&=\frac{1}{(2\pi)^p}\int_{2\pi \bR^p/\bZ^p}\phi_{X-\bmmu}^{n_i}(\bmt)e^{-\ri \langle \bmt, \bmm_i-n_i\bmmu\rangle}\rd \bmt,
\end{split}\end{align*}
where $\phi_X(\bmt)$ and $\phi_{X-\bmmu}(\bmt)$ are the characteristic functions of $X$ and $X-\bmmu$ respectively. 
We recall the domains $B_j(\delta)$ for $j=0,1,2,\cdots, p-1$ from \eqref{e:domainB}. Thanks to Proposition \ref{p:cfbound}, the characteristic function $|\phi^n_{X-\bmmu}(\bmt)|$ is exponentially small outside those sets $B_j(\delta)$.
\begin{align}\begin{split}\label{e:integ1}
\bP(S_{n_i}=\bmm_i)
&=\frac{1}{(2\pi)^p}\sum_{j=0}^{p-1}\int_{2\pi B_j(\delta)}\phi_{X-\bmmu}^{n_i}(\bmt)e^{-\ri \langle \bmt, \bmm_i-n_i\bmmu\rangle}\rd \bmt + e^{-c(\delta)n_i/p^3}\\
&=\frac{p}{(2\pi)^p}\int_{2\pi B_0(\delta)}\phi_{X-\bmmu}^{n_i}(\bmt)e^{-\ri \langle \bmt, \bmm_i-n_i\bmmu\rangle}\rd \bmt + e^{-c(\delta)n_i/p^3},
\end{split}\end{align}
where we used the fact that $\sum_{j=0}^{p-1}jm_{ij}\equiv 0\Mod p$, and the integrand is translation invariant by vectors $2\pi(0,1/p, 2/p,\cdots, p-1/p)\bZ$.
For any $\bmt\in  B_0(\delta)$, by definition there exists $\bmx\in \bR^{p-1}$ with $\|\bmx\|^2_2\leq \delta$ and $y\in [0,2\sqrt{p}\pi]$, such that $\bmt=Q(\bmx,y)=O\bmx+(y/\sqrt{p})\bm1$. By a change of variable, we can rewrite \eqref{e:integ1} as
\begin{align}\begin{split}\label{e:integ2}
&\phantom{{}={}}\frac{p}{(2\pi)^p}\int_{2\pi B_0(\delta)}\phi_{X-\bmmu}^{n_i}(\bmt)e^{-\ri \langle \bmt, \bmm_i-n_i\bmmu\rangle}\rd \bmt\\
&=\frac{p}{(2\pi)^{p}}\int_{\{\bmx\in \bR^{p-1}:\|\bmx\|^2_2\leq \delta\}\times [0,2\sqrt{p}\pi]}\phi_{X-\bmmu}^{n_i}(Q(\bmx,y))e^{-\ri \langle Q(\bmx,y), \bmm_i-n_i\bmmu\rangle}\rd \bmx\rd y\\
&=\frac{p^{3/2}}{(2\pi)^{p-1}}\int_{\{\bmx\in \bR^{p-1}:\|\bmx\|^2_2\leq \delta\}}\phi_{X-\bmmu}^{n_i}(O\bmx)e^{-\ri \langle O\bmx, \bmm_i-n_i\bmmu\rangle}\rd \bmx,
\end{split}\end{align}
where we used that $\langle \bm1, X-\bmmu\rangle=0$ and $\langle\bm1, \bmm_i-n_i\bmmu\rangle=0$. We recall the estimate of the characteristic function from \eqref{e:cf},
\begin{align}\begin{split}\label{e:cfu}
\phi_{X-\bmmu}(O\bmx)
=1-\frac{d}{2p}\|\bmx\|_2^2+\OO\left(\frac{\|\bmx\|_2^3}{p}\right).
\end{split}\end{align}
Fix a large constant $\fc$, which will be chosen later. For $\fc p^4\ln n/n\leq \|\bmx\|_2^2\leq \delta$, we have
\begin{align}\label{e:integ3}
|\phi_{X-\bmmu}(O\bmx)|^{n_i}\leq \exp\left\{-\left(\frac{\fc d}{2}+\oo(1)\right)p^2\ln n\right\},
\end{align}
which turns out to be negligible provided $\fc$ is large enough. In the following we will restrict the integral \eqref{e:integ2} on the domain 
$\{\bmx\in \bR^{p-1}:\|\bmx\|^2_2\leq \fc p^4\ln n/n\}$. From \eqref{e:cfu}, on the domain $\{\bmx\in \bR^{p-1}:\|\bmx\|^2_2\leq \fc p^4 \ln n/n\}$, we have 
\begin{align*}
\phi_{X-\bmmu}^{n_i}(O\bmx)=\left(1+\OO\left(\frac{p^4(\ln n)^{3/2}}{n^{1/2}}\right)\right)e^{-\frac{dn_i}{2p}\|\bmx\|^2_2},
\end{align*}
and
\begin{align}\begin{split}\label{e:integ4}
&\phantom{{}={}}\frac{p^{3/2}}{(2\pi)^{p-1}}\int_{\{\bmx\in \bR^{p-1}:\|\bmx\|^2_2\leq \fc p^4\ln n/n\}}\phi_{X-\bmmu}^{n_i}(O\bmx)e^{-\ri \langle O\bmx, \bmm_i-n_i\bmmu\rangle}\rd \bmx\\
&=\left(1+\OO\left(\frac{p^4(\ln n)^{3/2}}{n^{1/2}}\right)\right)\frac{p^{3/2}}{(2\pi)^{p-1}}\int_{\{\bmx\in \bR^{p-1}:\|\bmx\|^2_2\leq \fc p^4\ln n/n\}}e^{-\frac{dn_i}{2p}\|\bmx\|^2_2}e^{-\ri \langle \bmx, O^t(\bmm_i-n_i\bmmu)\rangle}\rd \bmx\\
&=\left(1+\OO\left(\frac{p^4(\ln n)^{3/2}}{n^{1/2}}\right)\right)\frac{p^{3/2}}{(2\pi)^{p-1}}\int_{ \bR^{p-1}}e^{-\frac{dn_i}{2p}\|\bmx\|^2_2}e^{-\ri \langle \bmx, O^t(\bmm_i-n_i\bmmu)\rangle}\rd \bmx+\OO\left(e^{-\left(\frac{\fc d}{2}+\oo(1)\right)p^2\ln n}\right)\\
&=\left(1+\OO\left(\frac{p^4(\ln n)^{3/2}}{n^{1/2}}\right)\right)\frac{p^{3/2}}{(2\pi)^{p-1}}\int_{ \bR^{p-1}}e^{-\frac{dn_i}{2p}\|\bmx\|^2_2}e^{-\ri \langle \bmx, O^t(\bmm_i-n_i\bmmu)\rangle}\rd \bmx+\OO\left(e^{-\left(\frac{\fc d}{2}+\oo(1)\right)p^2\ln n}\right)\\
&=\left(1+\OO\left(\frac{p^4(\ln n)^{3/2}}{n^{1/2}}\right)\right)p^{3/2}
\left(\frac{p}{2\pi d n_i}\right)^{(p-1)/2}e^{-\frac{n_i pd}{2}\left\|O^t \left(\frac{\bmm_i}{d n_i}-\frac{\bmmu}{d}\right)\right\|_2^2}
+\OO\left(e^{-\left(\frac{\fc d}{2}+\oo(1)\right)p^2\ln n}\right)\\
&=\left(1+\OO\left(\frac{p^4(\ln n)^{3/2}}{n^{1/2}}\right)\right)p^{3/2}
\left(\frac{p^2}{2\pi d n}\right)^{(p-1)/2}e^{-\frac{nd}{2}\sum_{j=0}^{p-1}\left(\frac{\fm_{ij}}{\fn_i}-\frac{1}{p}\right)^2}
+\OO\left(e^{-\left(\frac{\fc d}{2}+\oo(1)\right)p^2\ln n}\right).
\end{split}\end{align}
By our definition that $M$ is equidistributed, 
$\sum_{i,j=0}^{p-1}(\fm_{ij}-1/p^2)^2\leq \fb\ln n/n$, and by the AM-GM inequlity
$\sum_{j=0}^{p-1}(\fn_j-1/p)^2\leq \fb p\ln n/n$.  We can rewrite the exponent in \eqref{e:integ4} as
\begin{align}\label{e:integ5}
-\frac{nd}{2}\sum_{j=0}^{p-1}\left(\frac{\fm_{ij}}{\fn_i}-\frac{1}{p}\right)^2
=-\frac{dnp}{2}\left(p\sum_{j=0}^{p-1}\left(\fm_{ij}-\frac{1}{p^2}\right)^2-\left(\fn_i-\frac{1}{p}\right)^2\right)+\OO\left(\frac{p^3(\ln n)^{3/2}}{n^{1/2}}\right).
\end{align}
It follows by combining \eqref{e:integ4} and \eqref{e:integ5},  
\begin{align}\begin{split}\label{e:integ6}
&\prod_{i=0}^{p-1}\bP(X_1+X_2+\cdots+X_{n_i}=(m_{i0},m_{i1},\cdots, m_{ip-1}))=\OO\left(e^{-\left(\frac{\fc d}{2}+\oo(1)\right)p^2\ln n}\right)\\
&+\left(1+\OO\left(\frac{p^4(\ln n)^{3/2}}{n^{1/2}}\right)\right)p^{3p/2}
\left(\frac{p^2}{2\pi d n}\right)^{(p^2-p)/2}e^{-\frac{dnp}{2}\left(p\sum_{i,j=0}^{p-1}\left(\fm_{ij}-\frac{1}{p^2}\right)^2-\sum_{i=0}^{p-1}\left(\fn_i-\frac{1}{p}\right)^2\right)}.
\end{split}\end{align}
The exponents in \eqref{e:firstfactor} and \eqref{e:integ6} cancel
\begin{align}
\begin{split}
\label{e:integ7}
&\phantom{{}={}}\frac{dnp^2}{4}\sum_{i,j=0}^{p-1}\left(\fm_{ij}-\frac{1}{p^2}\right)^2-\frac{np}{2}\sum_{j=0}^{p-1}\left(\fn_j-\frac{1}{p}\right)^2-\frac{dnp}{2}\left(p\sum_{i,j=0}^{p-1}\left(\fm_{ij}-\frac{1}{p^2}\right)^2-\sum_{i=0}^{p-1}\left(\fn_i-\frac{1}{p}\right)^2\right)\\
&=-\frac{dnp^2}{4}\sum_{i,j=0}^{p-1}\left(\fm_{ij}-\frac{1}{p^2}\right)^2+\frac{(d-1)np}{2}\sum_{j=0}^{p-1}\left(\fn_j-\frac{1}{p}\right)^2.
\end{split}\end{align}
Therefore, by combining the estimates \eqref{e:firstfactor}, \eqref{e:integ6} and \eqref{e:integ7}, we conclude that for any equidistributed $p\times p$ symmetric matrices $M=[m_{ij}]_{0\leq i,j\leq p-1}\in \cM$,
\begin{align}\begin{split}\label{e:integ8}
&\frac{1}{|\mathsf{G}_{n,d}|}\sum_{\bmv\in \cS(n_0(M),n_1(M),\cdots, n_{p-1}(M))} |\{\cG\in \mathsf{G}_{n,d}: A(\cG)\bmv=\bm0\}|=\OO\left(e^{-\left(\frac{\fc d}{2}-\frac{\fb d}{4}+\oo(1)\right)p^2\ln n}\right)\\
&+\left(1+\OO\left(\frac{p^4(\ln n)^{3/2}}{n^{1/2}}\right)\right)\frac{2^{(p-1)/2}p^{(p^2+3p)/2}}{d^{(p^2-p)/4}(2\pi n)^{(p^2+p-2)/4}}
e^{-\frac{dnp^2}{4}\sum_{i,j=0}^{p-1}\left(\fm_{ij}-\frac{1}{p^2}\right)^2+\frac{(d-1)np}{2}\sum_{j=0}^{p-1}\left(\fn_j-\frac{1}{p}\right)^2}.
\end{split}\end{align}
For the first term on the righthand side of \eqref{e:integ8}, we notice that the total number of $p\times p$ symmetric matrices in $\cE$ is bounded by $e^{p^2\ln n}$, 
\begin{align}\label{e:firstterm}
\sum_{M\in \cE}\OO\left(e^{-\left(\frac{\fc d}{2}-\frac{\fb d}{4}+\oo(1)\right)p^2\ln n}\right)
\leq \OO\left(e^{-\left(\frac{\fc d}{2}-\frac{\fb d}{4}-1+\oo(1)\right)p^2\ln n}\right),
\end{align}
which is negligible provided $\fc$ is large enough.

For the second term on the righthand side of \eqref{e:integ8}, we denote $\tilde \cE$ the set of $p\times p$ symmetric matrices $M=[m_{ij}]_{0\leq i,j\leq p}$ such that 
\begin{enumerate}
\item$m_{ij}=m_{ji}\in \bZ_{\geq 0}$ for $0\leq i,j\leq p-1$ and $2|m_{ii}$ for $0\leq i\leq p-1$.
\item $\sum_{i,j=0}^{p-1}m_{ij}=dn$, and $\sum_{i,j=0}^{p-1}(m_{ij}/(dn)-1/p^2)^2\leq \fb\ln n/n$.
\end{enumerate}
The set $\cE$ is a subset of $\tilde \cE$ with the extra constraints: $\sum_{j=0}^{p-1}m_{ij}\equiv 0\Mod d$, and $\sum_{j=0}^{p-1}jm_{ij}\equiv 0 \Mod p$ for $i=0,1,\cdots, p-1$. In the following we prove that the sum over the set $\cE$ in \eqref{e:explclt} can be replaced by the sum over the set $\tilde \cE$ with a negligible error. We concentrate on the case $p$ is odd, and remark the modification for $p=2$ case later.

By our assumption $\gcd (p,d)=1$, for any vectors $\bmr=(r_0,r_1,\cdots, r_{p-1})\in \{0,1,\cdots,d-1\}^{p}$ with $r_0+r_1+\cdots+r_{p-1}\equiv 0 \Mod d$, and $\bms=(s_0,s_1,\cdots, s_{p-1})\in \{0,1,\cdots, p-1\}^p$, there exists a $p\times p$ matrix $[\Delta m^{(\bmr, \bms)}_{ij}]_{0\leq i,j\leq p-1}$ with entries size $\OO(p)$ (not unique), such that 
\begin{enumerate}
\item$\Delta m^{(\bmr, \bms)}_{ij}=\Delta m^{(\bmr,\bms)}_{ji}\in \bZ_{\geq 0}$ for $0\leq i,j\leq p-1$ and $2|\Delta m^{(\bmr,\bms)}_{ii}$ for $0\leq i\leq p-1$.
\item $\sum_{i,j=0}^{p-1}\Delta m^{(\bmr,\bms)}_{ij}=0$.
\item 
$\sum_{j=0}^{p-1}\Delta m^{(\bmr,\bms)}_{ij}\equiv r_i\Mod d$, $\sum_{j=0}^{p-1}j\Delta m_{ij}^{(\bmr,\bms)}\equiv s_i \Mod p$, for $i=0,1,\cdots, p-1$.
\end{enumerate}
For the term on the righthand side of \eqref{e:integ8} corresponding to $M=[m_{ij}]_{0\leq i,j\leq p-1}\in \cE$, we can rewrite it as an average of terms corresponding to $M^{(\bmr, \bms)}=[m^{(\bmr,\bms)}_{ij}\deq m_{ij}+\Delta m_{ij}^{(\bmr, \bms)}]_{0\leq i,j\leq p-1}\in \tilde \cE$, for $\bmr=(r_0,r_1,\cdots, r_{p-1})\in \{0,1,\cdots,d-1\}^{p}$ with $r_0+r_1+\cdots+r_{p-1}\equiv 0 \Mod d$, and $\bms=(s_0,s_1,\cdots, s_{p-1})\in \{0,1,\cdots, p-1\}^p$, 
\begin{align}\begin{split}\label{e:sumtE0}
&\phantom{{}={}}d^{p-1}p^p\left(1+\OO\left(\frac{p^4(\ln n)^{1/2}}{n^{1/2}}\right)\right)e^{-\frac{dnp^2}{4}\sum_{i,j=0}^{p-1}\left(\frac{m_{ij}}{dn}-\frac{1}{p^2}\right)^2+\frac{(d-1)np}{2}\sum_{j=0}^{p-1}\left(\frac{n_i}{n}-\frac{1}{p}\right)^2}\\
&=\sum_{\bmr\in \{0,1,\cdots, d-1\}^p\atop \sum_{i=0}^{p-1}r_i\equiv 0\Mod d}
\sum_{\bms\in \{0,1,\cdots, p-1\}^p}
e^{-\frac{dnp^2}{4}\sum_{i,j=0}^{p-1}\left(\frac{m_{ij}^{(\bmr,\bms)}}{dn}-\frac{1}{p^2}\right)^2+\frac{(d-1)np}{2}\sum_{j=0}^{p-1}\left(\frac{n_i^{(\bmr,\bms)}}{n}-\frac{1}{p}\right)^2},
\end{split}\end{align}
where $n_i^{(\bmr,\bms)}=\sum_{j=0}^{p-1}m_{ij}^{(\bmr,\bms)}$, for $i=0,1,\cdots, p-1$. We sum \eqref{e:sumtE0} over all the $p\times p$ symmetric matrices $M=[m_{ij}]_{0\leq i,j\leq p-1}\in\cE$, and get
\begin{align}\begin{split}\label{e:sumtE}
&\phantom{{}={}}d^{p-1}p^p\sum_{M\in \cE}e^{-\frac{dnp^2}{4}\sum_{i,j=0}^{p-1}\left(\fm_{ij}-\frac{1}{p^2}\right)^2+\frac{(d-1)np}{2}\sum_{j=0}^{p-1}\left(\fn_j-\frac{1}{p}\right)^2}\\
&=\left(1+\OO\left(\frac{p^4(\ln n)^{1/2}}{n^{1/2}}\right)\right)\sum_{M\in \tilde \cE}e^{-\frac{dnp^2}{4}\sum_{i,j=0}^{p-1}\left(\fm_{ij}-\frac{1}{p^2}\right)^2+\frac{(d-1)np}{2}\sum_{j=0}^{p-1}\left(\fn_j-\frac{1}{p}\right)^2}.
\end{split}\end{align}

We remark that if $p=2$, we have instead that 
\begin{align*}\begin{split}
&\phantom{{}={}}d^{p-1}p^{p-1}\sum_{M\in \cE}e^{-\frac{dnp^2}{4}\sum_{i,j=0}^{p-1}\left(\fm_{ij}-\frac{1}{p^2}\right)^2+\frac{(d-1)np}{2}\sum_{j=0}^{p-1}\left(\fn_j-\frac{1}{p}\right)^2}\\
&=\left(1+\OO\left(\frac{p^4(\ln n)^{1/2}}{n^{1/2}}\right)\right)\sum_{M\in \tilde \cE}e^{-\frac{dnp^2}{4}\sum_{i,j=0}^{p-1}\left(\fm_{ij}-\frac{1}{p^2}\right)^2+\frac{(d-1)np}{2}\sum_{j=0}^{p-1}\left(\fn_j-\frac{1}{p}\right)^2},
\end{split}\end{align*}
which differs from \eqref{e:sumtE} by a factor of $2$. As a consequence, this leads to $
\sum_{\bmv\in \bF_p^n\setminus{\bm0}}|\{\cG\in \mathsf{G}_{n,d}: A(\cG)\bmv=\bm0\}|= (2+\oo(1))|\mathsf{G}_{n,d}|
$.

In the following we estimate the sum in \eqref{e:sumtE}. The set of points $[m_{ij}/dn-1/p^2]_{0\leq i,j\leq p-1}$ for $M=[m_{ij}]_{0\leq i,j\leq p-1}\in \tilde \cE$ is a subset of a lattice in ${\rm Sym}_p^0$, the Hilbert space of $p\times p$ real symmetric matrices with total sum zero, and inner product $\langle A, B\rangle=\Tr AB$. A set of base for this lattice is given by
\begin{align*}
(e_{ij}+e_{ji}-2e_{00})/dn, \quad 0\leq i<j\leq p-1, \quad (2e_{ii}-2e_{00})/dn, \quad 1\leq i\leq p-1.
\end{align*}
The volume of the fundamental domain is $2^{(p^2+3p-4)/4}p(dn)^{-(p^2+p-2)/2}$. By viewing \eqref{e:sumtE} as a Riemann sum, we can rewrite it as an integral on the space ${\rm Sym}_p^0$.
\begin{align}
\begin{split}\label{e:gaussianI}
&\phantom{{}={}}2^{(p^2+3p-4)/4}p(dn)^{-(p^2+p-2)/2}\sum_{M\in \tilde \cE}e^{-\frac{dnp^2}{4}\sum_{i,j=0}^{p-1}\left(\fm_{ij}-\frac{1}{p^2}\right)^2+\frac{(d-1)np}{2}\sum_{j=0}^{p-1}\left(\fn_j-\frac{1}{p}\right)^2}\\
&=\left(1+\OO\left(\frac{p^4(\ln n)^{1/2}}{n^{1/2}}\right)\right)
\int_{A\in {\rm Sym}_p^0:\|A\|_2^2\leq \fb\ln n/n}e^{-\frac{dnp^2}{4}\Tr A^2+\frac{(d-1)np}{2}\bm1^t A^2 \bm1}\rd \vol(A).
\end{split}
\end{align}
We can rewrite the exponent as a quadratic form on the space ${\rm Sym}_p^{0}$,
\begin{align*}
-dnp^2\Tr A^2/4+(d-1)np\bm1^t A^2 \bm1/2
=\langle A, \cL(A)\rangle,
\end{align*}
where the self-adjoint operator $\cL: {\rm Sym}_p^0\mapsto {\rm Sym}_p^0$ is given by $\cL(A)=-dnp^2A/4+(d-1)npA \bm1\bm1^t/4+(d-1)np\bm 1\bm1^tA /4$. The self-adjoint operator $\cL$ is diagonalized by 
\begin{enumerate}
\item If $A\in{\rm Sym}_p^0$, with row sums and column sums zero, then $\cL(A)=(-dnp^2/4)A$. The total dimension of such matrices is $p(p-1)/2$.
\item If $A=[a_i+a_j]_{0\leq i,j\leq p-1}$ for some vector $\bma=(a_0,a_1,\cdots, a_{p-1})\in \bR^{p}$ with $a_0+a_1+\cdots+a_{p-1}=0$, then $\cL(A)=(-np^2/4)A$. The total dimension of such matrices is $p-1$.
\end{enumerate}
From the discussion above, using the eigenvectors of the self-adjoint operator $\cL$ as a base, the integral \eqref{e:gaussianI} decomposes into a product of Gaussian integrals, which can be estimated explicitly.
\begin{align}\begin{split}\label{e:integral}
&\phantom{{}={}}\int_{A\in {\rm Sym}_p^0:\|A\|_2^2\leq \fb\ln n/n}e^{-\frac{dnp^2}{4}\Tr A^2+\frac{(d-1)np}{2}\bm1^t A^2 \bm1}\rd \vol(A)\\
&=\int_{A\in {\rm Sym}_p^0}e^{-\frac{dnp^2}{4}\Tr A^2+\frac{(d-1)np}{2}\bm1^t A^2 \bm1}\rd \vol(A)+\OO\left(e^{-\left(\frac{\fb}{4}+\oo(1)\right)p^2\ln n}\right)\\
&=\left(\frac{4\pi}{dnp^2}\right)^{(p^2-p)/4}\left(\frac{4\pi}{np^2}\right)^{(p-1)/2}+\OO\left(e^{-\left(\frac{\fb}{4}+\oo(1)\right)p^2\ln n}\right).
\end{split}\end{align}
We can estimate the total contribution in \eqref{e:equ} from the second term on the righthand side of \eqref{e:integ8}, by combining the estimates \eqref{e:sumtE}, \eqref{e:gaussianI} and \eqref{e:integral}
\begin{align}\begin{split}\label{e:secondterm}
&\phantom{{}={}}\left(1+\OO\left(\frac{p^4(\ln n)^{3/2}}{n^{1/2}}\right)\right)\frac{2^{(p-1)/2}p^{(p^2+3p)/2}}{d^{(p^2-p)/4}(2\pi n)^{(p^2+p-2)/4}}
\sum_{M\in \cE}e^{-\frac{dnp^2}{4}\sum_{i,j=0}^{p-1}\left(\fm_{ij}-\frac{1}{p^2}\right)^2+\frac{(d-1)np}{2}\sum_{j=0}^{p-1}\left(\fn_j-\frac{1}{p}\right)^2}\\
&=\left(1+\OO\left(\frac{p^4(\ln n)^{3/2}}{n^{1/2}}\right)\right)
+\OO\left(e^{-\left(\frac{\fb}{4}+\oo(1)\right)p^2\ln n+\frac{p^2+p-2}{4}\ln (np^2)}\right)
=1+\OO\left(\frac{p^4(\ln n)^{3/2}}{n^{1/2}}\right),
\end{split}\end{align}
provided $\fb$ is large enough. Proposition \ref{p:lclt} follows from combining \eqref{e:firstterm} and \eqref{e:secondterm}.
\end{proof}

\subsection{Large deviation estimate}
\label{subs:ldp}

In this section, we show that the sum of terms in \eqref{e:keyest} corresponding to non-equidistributed $p\times p$ symmetric matrices $M=[m_{ij}]_{0\leq i,j\leq p-1}$ is small. 
\begin{proposition}\label{p:ldp}
Let $d\geq 3$ be a fixed integer, and a prime number $p$ such that $\gcd(p,d)=1$ and $p\ll n^{(d-2)/(5d-6)}$. Then for $n$ sufficiently large,
\begin{align}\label{e:noneq}
\frac{1}{|\mathsf{G}_{n,d}|}\sum_{M \in \cN}\sum_{\bmv\in \cS(n_0(M),n_1(M),\cdots, n_{p-1}(M))} |\{\cG\in \mathsf{G}_{n,d}: A(\cG)\bmv=\bm0\}|\leq \frac{\OO(p^{2d})}{n^{(d-2)}}.
\end{align}
\end{proposition}
Thanks to Proposition \ref{p:walkrepu}, we have
\begin{align}\begin{split}\label{e:expldp}
&\phantom{{}={}}\frac{1}{|\mathsf{G}_{n,d}|}\sum_{M \in \cN}\sum_{\bmv\in \cS(n_0(M),n_1(M),\cdots, n_{p-1}(M))} |\{\cG\in \mathsf{G}_{n,d}: A(\cG)\bmv=\bm0\}|\\
&=\frac{2^{nd/2}(nd/2)!}{(nd)!}\sum_{M\in \cN}{n\choose n_0(M),n_1(M),\cdots, n_{p-1}(M)}\prod_{0\leq i<j\leq p-1}m_{ij}!\prod_{i=0}^{p-1}\frac{m_{ii}!}{2^{m_{ii}/2}(m_{ii}/2)!}\\
& \phantom{{}={}}\times p^{(d-1)n}\prod_{i=0}^{p-1}\bP(X_1+X_2+\cdots+X_{n_i}=(m_{i0},m_{i1},\cdots, m_{ip-1}))
\end{split}\end{align}
where $X_1,X_2,\cdots, X_{n_i}$ are independent copies of $X$, which is uniform distributed over $\cU_{d,p}$ as defined in \eqref{e:defcU}.
In the rest of the proof, we simply write $n_i(M)$ as $n_i$ for $0\leq i\leq p-1$. For an non-equidistributed $p\times p$ symmetric matrix $M=[m_{ij}]_{0\leq i,j\leq p-1}$, we denote $\fm_{ij}=m_{ij}/(dn)$ for $i,j=0,1,\cdots, p-1$, and $\fn_i=n_i/n$ for $i=0,1,\cdots, p-1$. Then we have $\fn_i=\sum_{j=0}^{p-1}\fm_{ij}$ for $i=0,1,\cdots, p-1$. Moreover, by our definition of non-equidistributed, 
$\sum_{i,j=0}^{p-1}(\fm_{ij}-1/p^2)^2\geq \fb\ln n/n$. We estimate the first factor on the righthand side of \eqref{e:expldp} using Stirling's formula,
\begin{align}\begin{split}\label{e:firstfactor2}
&\phantom{{}={}}\frac{2^{nd/2}(nd/2)!}{(nd)!}{n\choose n_0,n_1,\cdots, n_{p-1}}\prod_{0\leq i<j\leq p-1}m_{ij}!\prod_{i=0}^{p-1}\frac{m_{ii}!}{2^{m_{ii}/2}(m_{ii}/2)!}\\
&\leq e^{\OO(p^2\ln n)}
\exp\left\{\frac{dn}{2}\sum_{i,j=0}^{p-1}{\fm_{ij}}\ln \fm_{ij}-n\sum_{j=0}^{p-1}\fn_j\ln \fn_j \right\}.
\end{split}\end{align}
For the random walk term in \eqref{e:expldp}, we have the following large deviation bound
\begin{align*}
\bP(X_1+X_2+\cdots+X_{n_i}=(m_{i0},m_{i1},\cdots, m_{ip-1}))
\leq \exp\left\{n\inf_{\bmt_i\in \bR^p} \fn_i\log \bE[e^{\langle \bmt_i, X\rangle}]-d\langle \bmt_i, \bm\fm_i\rangle\right\},
\end{align*}
where $\fm_i=(\fm_{i0}, \fm_{i1},\cdots, \fm_{ip-1})$.
Thus we get that
\begin{align*}
\begin{split}
&\phantom{{}={}}\frac{1}{|\mathsf{G}_{n,d}|}\sum_{M \in \cN}\sum_{\bmv\in \cS(n_0(M),n_1(M),\cdots, n_{p-1}(M))} |\{\cG\in \mathsf{G}_{n,d}: A(\cG)\bmv=\bm0\}|\\
&\leq 
\sum_{M \in \cN}\sum_{\bmv\in \cS(n_0(M),n_1(M),\cdots, n_{p-1}(M))} 
e^{\OO(p^2\ln n)}e^{nI(\fm_0,\fm_1, \cdots, \fm_{p-1})}
\end{split}\end{align*}
where the rate function is given by
\begin{align}
\begin{split}\label{e:ratefu}
I(\fm_0,\fm_1, \cdots, \fm_{p-1})
&=
(d-1)\ln p+\frac{d}{2}\sum_{i,j=0}^{p-1}{\fm_{ij}}\ln \fm_{ij}-\sum_{j=0}^{p-1}\fn_j\ln \fn_j\\
&+\sum_{j=0}^{p-1}\inf_{\bmt_j\in \bR^p} \fn_j\log \bE[e^{\langle \bmt_j, X\rangle}]-d\langle \bmt_j, \bm\fm_j\rangle.
\end{split}\end{align}

\begin{proposition}\label{p:ldpboundu}
Let $d\geq 3$ be a fixed integer, and a prime number $p$ such that $\gcd(p,d)=1$. The rate function as defined in \eqref{e:ratefu} satisfies: for any small $\delta> 0$, there exists a constant $c(\delta)>0$, such that
\begin{align}\label{e:rateboundu}
I(\fm_0, \fm_1,\cdots, \fm_{p-1})\leq -\frac{c(\delta)}{p^2},
\end{align}
unless
$ \max_{0\leq i,j\leq p-1}|\fm_{ij}-1/p^2|\leq \delta/p^2$, or
$\fm_{00}\geq 1-\delta/p$.
\end{proposition}

\begin{proof}
We take $\bmt_j=(d-1)/d((\ln (\fm_{j0}/\fn_j), \ln (\fm_{j1}/\fn_j),\cdots, \ln (\fm_{j p-1}/\fn_j))+\ln p)$ in \eqref{e:ratefu}, the rate function $I$ is upper bounded by 
\begin{align}\begin{split}\label{e:sharprateu}
I(\fm_0,\fm_1, \cdots, \fm_{p-1})
&\leq \frac{d-2}{2}\sum_{i,j=0}^{p-1}\fm_{ij}\ln \frac{\fn_i\fn_j}{\fm_{ij}}+\sum_{i=0}^{p-1}\fn_i\log \sum_{j=1}^{p^{d-1}}\prod_{k=0}^{p-1}\left(\frac{\fm_{ik}}{\fn_i}\right)^{\frac{d-1}{d}\bmw_j(k)}.
\end{split}\end{align}
In the following, we prove that there exists a constant $c(\delta)$
\begin{align}\label{e:ratebound2u}
\frac{d-2}{2}\sum_{i,j=0}^{p-1}\fm_{ij}\ln \frac{\fn_i\fn_j}{\fm_{ij}}+\sum_{i=0}^{p-1}\fn_i\log \sum_{j=1}^{p^{d-1}}\prod_{k=0}^{p-1}\left(\frac{\fm_{ik}}{\fn_i}\right)^{\frac{d-1}{d}\bmw_j(k)}\leq -\frac{c(\delta)}{p^2}, 
\end{align}
unless $ \max_{0\leq i,j\leq p-1}|\fm_{ij}-1/p^2|\leq \delta/p^2$, or
$\fm_{00}\geq 1-\delta/p$. Then the claim \eqref{e:rateboundu} follows.

Thanks to Proposition \ref{p:ldpbound}, for any $\fn_i>0$ and $\varepsilon>0$ sufficiently small, 
\begin{align}\label{e:ldpboundcopy}
\log \sum_{j=1}^{p^{d-1}}\prod_{k=0}^{p-1}\left(\frac{\fm_{ik}}{\fn_i}\right)^{\frac{d-1}{d}\bmw_j(k)}\leq -\frac{c(\varepsilon)}{p},
\end{align}
unless $ \max_{0\leq k\leq p-1}|\fm_{ik}/\fn_i-1/p|\leq \varepsilon/p$, or
$\fm_{i0}/\fn_i\geq 1-\varepsilon/p$. We decompose $\{0,1,\cdots, p-1\}=I_1\cup I_2\cup I_3$, where
\begin{align}\begin{split}\label{e:decompose}
&I_1=\{0\leq i\leq p-1: \max_{0\leq k\leq p-1}|\fm_{ik}/\fn_i-1/p|\leq \varepsilon/p\},\\
&I_2=\{0\leq i\leq p-1: \fm_{i0}/\fn_i\geq 1-\varepsilon/p\}, \quad I_3=\{0,1,\cdots, p-1\}\setminus (I_1\cup I_2).
\end{split}\end{align}
Thanks to \eqref{e:ldpboundcopy}, we have
\begin{align*}
\sum_{i\in I_3}\fn_i \log \sum_{j=1}^{p^{d-1}}\prod_{k=0}^{p-1}\left(\frac{\fm_{ik}}{\fn_i}\right)^{\frac{d-1}{d}\bmw_j(k)}\leq -\frac{c(\varepsilon)}{p}\sum_{i\in I_3}\fn_i.
\end{align*}
Therefore, if $\sum_{i\in I_3}\fn_i\geq \varepsilon/p$ then \eqref{e:ratebound2u} holds. In the following we assume that $\sum_{i\in I_3}\fn_i\leq \varepsilon/p$. There are several cases:
\begin{enumerate}
\item $I_2=\emptyset$: From the discussion above, we have $\sum_{i\in I_3}\fn_i\leq \varepsilon/p$ and $\sum_{i\in I_1}\fn_i\geq 1-\varepsilon/p$. If $I_3\neq \emptyset$, we fix any $k\in I_3$. By the definition of $I_1$, we have $\sum_{i\in I_1}\fn_i\leq \sum_{i\in I_1}p\fm_{ik}/(1-\varepsilon)\leq p\fn_k/(1-\varepsilon)\leq \varepsilon/(1-\varepsilon)$, which leads to a contradiction. Therefore $I_3=\emptyset$ and $I_1=\{0,1,2,\cdots,p-1\}$. Then for any $\fm_{ij}$ and $\fm_{i'j'}$, we have
\begin{align*}
\fm_{ij}\leq \frac{1+\varepsilon}{1-\varepsilon}\fm_{ii'}=\frac{1+\varepsilon}{1-\varepsilon}\fm_{i'i}\leq \frac{(1+\varepsilon)^2}{(1-\varepsilon)^2}\fm_{i'j'}.
\end{align*}
It follows that $|\fm_{ij}-1/p^2|\leq \delta/p^2$ by taking $4\varepsilon/(1-\varepsilon)^2
\leq \delta$.

\item $I_2\neq \emptyset$:
Thanks to \eqref{e:midbound}, if $i\in I_2$,  
\begin{align*}\begin{split}
&\phantom{{}={}}\fn_i\log \sum_{j=1}^{p^{d-1}}\prod_{k=0}^{p-1}\left(\frac{\fm_{ik}}{\fn_i}\right)^{\frac{d-1}{d}\bmw_j(k)}
\leq -(1+\oo(1))(d-1)\fn_i\left(1-\frac{\fm_{i0}}{\fn_i}\right)\\
&=-(1+\oo(1))(d-1)\left(\fn_i-\fm_{i0}\right)=-(1+\oo(1))(d-1)\sum_{1\leq k\leq p-1}\fm_{ik}.
\end{split}\end{align*}
Therefore, if $\sum_{i\in I_2}\sum_{1\leq k\leq p-1}\fm_{ik}\geq \varepsilon/p^2$ then we have that \eqref{e:ratebound2u} holds. In the following, we assume $\sum_{i\in I_2}\sum_{1\leq k\leq p-1}\fm_{ik}\leq \varepsilon/p^2$. There are several cases.
\begin{enumerate}
\item $0\in I_1$: By the definition of $I_1$, for any $0\leq k\leq p-1$, we have
\begin{align*}
(1-\varepsilon)\fn_0/p\leq\fm_{k0}=\fm_{0k}\leq (1+\varepsilon)\fn_0/p.
\end{align*}
Fix any $j\in I_2$. Again by the definition of $I_1$, $\sum_{i\in I_1\setminus\{0\}}\fn_i\leq \sum_{i\in I_1\setminus\{0\}}p\fm_{ij}/(1-\varepsilon/p)=\sum_{i\in I_1\setminus\{0\}}p\fm_{ji}/(1-\varepsilon/p)\leq \varepsilon/(p-\varepsilon)$. Therefore, $\sum_{1\leq i,k\leq p-1}\fm_{ik}\leq \sum_{i\in I_3}\fn_i+\sum_{i\in I_1\setminus\{0\}}\fn_i+\sum_{i\in I_2}\sum_{1\leq k\leq p-1}\fm_{ik}\leq \varepsilon/p+\varepsilon/(p-\varepsilon)+\varepsilon/p^2=\OO(\varepsilon)/p$. As a consequence, we get $\fm_{0k}=\fm_{k0}=(1+\OO(\varepsilon))/(2p-1)$, for any $0\leq k\leq p-1$, and $\fn_0=(1+\OO(\varepsilon))p/(2p-1)$ and $\fn_k=(1+\OO(\varepsilon))/(2p-1)$ for $1\leq k\leq p-1$. In this case, the first term on the lefthand side of \eqref{e:ratebound2u} is small
\begin{align*}
\frac{d-2}{2}\sum_{i,j=0}^{p-1}\fm_{ij}\ln \frac{\fn_i\fn_j}{\fm_{ij}}\leq (1+\OO(\varepsilon))\frac{d-2}{2}\left(\frac{\ln p}{2p-1}+\ln \frac{p}{2p-1}\right)+\OO\left(\frac{\varepsilon}{p}\ln \frac{\varepsilon}{p}\right),
\end{align*}
and \eqref{e:ratebound2u} holds.

\item $0\in I_2$: By the definition of $I_2$, we have $\sum_{i\in I_2\setminus\{0\}}\fn_i\leq \sum_{i\in I_2\setminus\{0\}}\fm_{i0}/(1-\varepsilon/p)\leq \sum_{1\leq k\leq p-1}\fm_{0k}/(1-\varepsilon/p)\leq \varepsilon/p(p-\varepsilon)$. By the definition of $I_1$, we have $\sum_{i\in I_1}\fn_i\leq \sum_{i\in I_1}p\fm_{i0}/(1-\varepsilon)\leq \sum_{1\leq k\leq p-1}p\fm_{0k}/(1-\varepsilon)\leq \varepsilon/p(1-\varepsilon)$. Combining the discussion above, we get
\begin{align*}\begin{split}
\fm_{00}
&\geq \fn_{0}-\varepsilon/p^2=1-\sum_{k=1}^{p-1}\fn_k-\varepsilon/p^2\\
&\geq 1-\varepsilon/p(p-\varepsilon)-\varepsilon/p(1-\varepsilon)-\varepsilon/p-\varepsilon/p^2.
\end{split}\end{align*}
It follows that $\fm_{00}\geq 1-\delta/p$ by taking $4\varepsilon/(1-\varepsilon)\leq \delta$.
\item $0\in I_3$: By the definition of $I_1$, we have $\sum_{i\in I_1}\fn_i\leq \sum_{i\in I_1}p\fm_{i0}/(1-\varepsilon)\leq \sum_{1\leq k\leq p-1}p\fm_{0k}/(1-\varepsilon)\leq p\fn_0/(1-\varepsilon)\leq \varepsilon/(1-\varepsilon)$. However, we also know that $\sum_{i\in I_3}\fn_i\leq \varepsilon/p$ and $\sum_{i\in I_2}\sum_{1\leq k\leq p-1}\fm_{ik}\leq \varepsilon/p^2$. This contradicts to the fact $\sum_{0\leq i\leq p-1}\fn_i=1$.

\end{enumerate}

\end{enumerate}


\end{proof}

\begin{proof}[Proof of Proposition \ref{p:ldp}]
We further decompose the set of non-equidistributed $p\times p$ symmetric matrices $M=[m_{ij}]_{0\leq i,j\leq p-1}$ into four classes:
\begin{enumerate}
\item $p\times p$ symmetric matrices $M=[m_{ij}]_{0\leq i,j\leq p-1}\in \cN$ with $\max_{0\leq i,j\leq p-1}|m_{ij}/(dn)-1/p^2|\leq \delta/p^2$.
\item  $p\times p$ symmetric matrices $M=[m_{ij}]_{0\leq i,j\leq p-1}\in \cN$ with $\fb p^3\ln n/n<|m_{00}/(dn)-1|\leq \delta/p$.
\item  $p\times p$ symmetric matrices $M=[m_{ij}]_{0\leq i,j\leq p-1}\in \cN$ with $|m_{00}/(dn)-1|\leq \fb p^3\ln n/n$.
\item The remaining non-equidistributed  $p\times p$ symmetric matrices.
\end{enumerate}

For the first class, $\max_{0\leq i,j\leq p-1}|m_{ij}/(dn)-1/p^2|\leq \delta/p^2$. The total number of such $p\times p$ symmetric matrices is $e^{\OO(p^2\ln n)}$. 
Given a $p\times p$ symmetric matrix $M=[m_{ij}]_{0\leq i,j\leq p-1}$ in the first class,
we will derive a more precise estimate of \eqref{e:ratebound2u}, by a perturbation argument. Let 
\begin{align*}\begin{split}
&\fm_{ij}=(1+\delta_{ij})/p^2,\quad i=0,1,\cdots, p-1, \quad j=0,1,\cdots, p-1.
\end{split}\end{align*}
where $\delta_{ij}=\delta_{ji}$ for $0\leq i<j\leq p-1$, $\sum_{i,j=0}^{p-1}\delta_{ij}=0$, $\max_{0\leq i,j\leq p-1}|\delta_{ij}|\leq \delta$, and $\sum_{ij}\delta_{ij}^2\geq \fb p^4 \ln n/n$. We denote,
\begin{align}\label{e:constraint2}
\fn_i=\sum_{j=0}^{p-1}\fm_{ij}=(1+\delta_i)/p,\quad \delta_i=\sum_{j=0}^{p-1}\delta_{ij}/p.
\end{align}
We use Taylor expansion, and rewrite the first term in \eqref{e:sharprateu} as
\begin{align}\begin{split}\label{e:t1}
\frac{d-2}{2}\sum_{i,j=0}^{p-1}\fm_{ij}\ln \frac{\fn_i\fn_j}{\fm_{ij}}
&=(d-2)\sum_{i=0}^{p-1}\fn_{i}\ln \fn_i
-\frac{d-2}{2}\sum_{i,j=0}^{p-1}\fm_{ij}\ln \fm_{ij}\\
&=(d-2)\sum_{i=0}^{p-1}\frac{1+\delta_i}{p}\ln\frac{1+\delta_i}{p}
-\frac{d-2}{2}\sum_{i,j=0}^{p-1}\frac{1+\delta_{ij}}{p^2}\ln \frac{1+\delta_{ij}}
{p^2}\\
&=(1+\OO(\delta))\left(\frac{d-2}{2p}\sum_{i=0}^{p-1}\delta_i^2-\frac{d-2}{4p^2}\sum_{i,j=0}^{p-1}\delta_{ij}^2\right).
\end{split}\end{align}
For the second term in \eqref{e:sharprateu}, similar to \eqref{e:sumexp} we have
\begin{align*}\begin{split}
&\phantom{{}={}}\log \frac{1}{p^{d-1}}\sum_{j=1}^{p^{d-1}}\prod_{k=0}^{p-1}\fm_{ik}^{\frac{d-1}{d}\bmw_j(k)}=\log \frac{1}{p^{d-1}}\sum_{j=1}^{p^{d-1}}\prod_{k=0}^{p-1}e^{\frac{d-1}{d}\bmw_j(k)\ln(1+\delta_{ik})}\\
&=\log \frac{1}{p^{d-1}}\sum_{j=1}^{p^{d-1}}e^{\frac{d-1}{d}\sum_{k=0}^{p-1}\bmw_j(k)\left(\delta_{ik}-(1+\OO(\delta)\frac{\delta_{ik}^2}{2}\right)}\\
&=\log \frac{1}{p^{d-1}}\sum_{j=1}^{p^{d-1}}1+\frac{d-1}{d}\sum_{k=0}^{p-1}\bmw_j(k)\left(\delta_{ik}-(1+\OO(\delta)\frac{\delta_{ik}^2}{2}\right)+(1+\OO(\delta))\frac{(d-1)^2}{2d^2}\left(\sum_{k=0}^{p-1}\bmw_j(k)\delta_{ik}\right)^2
\\
&=\log 1+(d-1)\delta_i-(1+\OO(\delta))\frac{d-1}{2p}\sum_{k=0}^{p-1}\delta_{ik}^2+\frac{1+\OO(\delta)}{p^{d-1}}\frac{(d-1)^2}{2d^2}\sum_{j=1}^{p^{d-1}}\left(\sum_{k=0}^{p-1}\bmw_j(k)\delta_{ik}\right)^2\\
&=(d-1)\delta_i-(1+\OO(\delta))\frac{d-1}{2p}\sum_{k=0}^{p-1}\delta_{ik}^2+\frac{1+\OO(\delta)}{p^{d-1}}\frac{(d-1)^2}{2d^2}\sum_{j=1}^{p^{d-1}}\left(\sum_{k=0}^{p-1}\bmw_j(k)\delta_{ik}\right)^2-\frac{(d-1)^2}{2}\delta_i^2.
\end{split}\end{align*}
We can rewrite the second term in \eqref{e:sharprateu} as
\begin{align}\begin{split}\label{e:t2}
&\phantom{{}={}}\sum_{i=0}^{p-1}\fn_i\log \sum_{j=1}^{p^{d-1}}\prod_{k=0}^{p-1}\left(\frac{\fm_{ik}}{\fn_i}\right)^{\frac{d-1}{d}\bmw_j(k)}
=-(d-1)\sum_{i=0}^{p-1}\fn_i\log \fn_i+\sum_{i=0}^{p-1}\fn_i\log \sum_{j=1}^{p^{d-1}}\prod_{k=0}^{p-1}\fm_{ik}^{\frac{d-1}{d}\bmw_j(k)}\\
&=-(d-1)\sum_{i=0}^{p-1}\frac{1+\delta_i}{p}\log \frac{1+\delta_i}{p}+\sum_{i=0}^{p-1}\frac{1+\delta_i}{p}\log \frac{1}{p^{2(d-1)}}\sum_{j=1}^{p^{d-1}}\prod_{k=0}^{p-1}e^{\frac{d-1}{d}\bmw_j(k)\ln(1+\delta_{ik})}\\
&=(1+\OO(\delta))\left(-\frac{d^2-3d+2}{2p}\sum_{i=0}^{p-1}\delta_i^2-\frac{d-1}{2p^2}\sum_{i,k=0}^{p-1}\delta_{ik}^2+\frac{1}{p^d}\frac{(d-1)^2}{2d^2}\sum_{i=0}^{p-1}\sum_{j=1}^{p^{d-1}}\left(\sum_{k=0}^{p-1}\bmw_j(k)\delta_{ik}\right)^2\right).
\end{split}\end{align}
We get the following estimate of \eqref{e:sharprateu} by combining \eqref{e:t1} and \eqref{e:t2},
\begin{align}\begin{split}\label{e:simple}
&\phantom{{}={}}\frac{d-2}{2}\sum_{i,j=0}^{p-1}\fm_{ij}\ln \frac{\fn_i\fn_j}{\fm_{ij}}+\sum_{i=0}^{p-1}\fn_i\log \sum_{j=1}^{p^{d-1}}\prod_{k=0}^{p-1}\left(\frac{\fm_{ik}}{\fn_i}\right)^{\frac{d-1}{d}\bmw_j(k)}\\
&=(1+\OO(\delta))\left(-\frac{(d-2)^2}{2p}\sum_{i=0}^{p-1}\delta_i^2-\frac{3d-4}{4p^2}\sum_{i,k=0}^{p-1}\delta_{ik}^2+\frac{1}{p^d}\frac{(d-1)^2}{2d^2}\sum_{i=0}^{p-1}\sum_{j=1}^{p^{d-1}}\left(\sum_{k=0}^{p-1}\bmw_j(k)\delta_{ik}\right)^2\right).
\end{split}\end{align}
We denote the $p\times p^{d-1}$ matrix $W=[\bmw_1,\bmw_2,\cdots, \bmw_{p^{d-1}}]$ and $p\times p$ matrix $\bm\delta=[\delta_{ij}]_{0\leq i,j\leq p-1}$. Then $\bm \delta \in {\rm Sym}_p^0$, the Hilbert space of $p\times p$ real symmetric matrices with total sum zero, and inner product $\langle A, B\rangle=\Tr AB$. We can rewrite \eqref{e:simple} as a quadratic form on the space ${\rm Sym}_p^{0}$
\begin{align}\begin{split}\label{e:constraint3u}
&\phantom{{}={}}\frac{d-2}{2}\sum_{i,j=0}^{p-1}\fm_{ij}\ln \frac{\fn_i\fn_j}{\fm_{ij}}+\sum_{i=0}^{p-1}\fn_i\log \sum_{j=1}^{p^{d-1}}\prod_{k=0}^{p-1}\left(\frac{\fm_{ik}}{\fn_i}\right)^{\frac{d-1}{d}\bmw_j(k)}\\
&=(1+\OO(\delta))\left(-\frac{(d-2)^2}{2p^3}\langle \bm\delta, \bm1 \bm1^t \bm\delta\rangle-\frac{3d-4}{4p^2}\langle \bm\delta, \bm\delta\rangle+\frac{1}{p^d}\frac{(d-1)^2}{2d^2}\langle \bm\delta, WW^t \bm\delta\rangle\right)\\
&=(1+\OO(\delta))\left(\frac{d^2-d-1}{2dp^3}\langle \bm\delta, \bm1 \bm1^t \bm\delta\rangle-\frac{d^2-2}{4dp^2}\langle \bm\delta, \bm\delta\rangle\right),
\end{split}\end{align}
where we used \eqref{e:cov}, 
\begin{align*}
WW^t=dp^{d-2}I_{p}+d(d-1)p^{d-3}\bm1\bm1^t.
\end{align*}
We can rewrite the first term in \eqref{e:constraint3u} as $\langle \bm\delta, \bm1 \bm1^t\bm\delta\rangle=\langle \bm\delta, \cL(\bm\delta)\rangle$ where  the self-adjoint operator $\cL: {\rm Sym}_p^0\mapsto {\rm Sym}_p^0$ is given by  $\cL(\bm\delta)=\bm1\bm1^t \bm \delta/2+\bm\delta \bm1\bm1^t /2$.
The self-adjoint operator $\cL$ is diagonalized by 
\begin{enumerate}
\item If $\bm\delta\in{\rm Sym}_p^0$, with row sums and column sums zero, then $\cL(\bm\delta)=\bm0$. The total dimension of such matrices is $p(p-1)/2$.
\item If $\bm\delta=[a_i+a_j]_{0\leq i,j\leq p-1}$ for some vector $\bma=(a_0,a_1,\cdots, a_{p-1})\in \bR^{p}$ with $a_0+a_1+\cdots+a_{p-1}=0$, then $\cL(\bm\delta)=p\bm\delta/2$. The total dimension of such matrices is $p-1$.
\end{enumerate}
It follows from the spectral decomposition of the self-adjoint operator $\cL$, we get
\begin{align*}
\langle \bm\delta, \bm1\bm1^t\bm\delta\rangle=\langle \bm\delta, \cL(\bm\delta)\rangle\leq p\langle \bm\delta,\bm\delta\rangle/2,
\end{align*}
and 
\begin{align*}
&\phantom{{}={}}\frac{d-2}{2}\sum_{i,j=0}^{p-1}\fm_{ij}\ln \frac{\fn_i\fn_j}{\fm_{ij}}+\sum_{i=0}^{p-1}\fn_i\log \sum_{j=1}^{p^{d-1}}\prod_{k=0}^{p-1}\left(\frac{\fm_{ik}}{\fn_i}\right)^{\frac{d-1}{d}\bmw_j(k)}
\leq -(1+\OO(\delta))\frac{d-1}{4dp^2}\langle \bm\delta,\bm\delta\rangle.
\end{align*}
The total contribution of terms in \eqref{e:noneq} satisfying $\max_{0\leq i,j\leq p-1}|m_{ij}/(dn)-1/p^2|\leq \delta/p^2$ is bounded by
\begin{align}\label{e:case1}
\exp\left\{-\left(\frac{\fb (d-1)}{4d}+\oo(1)\right)p^2\ln n+\OO(p^2\ln n)\right\}=\frac{\oo(1)}{n^{(d-2)}},
\end{align}
provided that we take $\fb$ sufficiently large.

For the second class, $\fb p^3\ln n/n<|m_{00}/(dn)-1|\leq \delta/p$. The total number of such $p\times p$ symmetric matrices is $e^{\OO(p^2\ln n)}$. Given a $p\times p$ symmetric matrix $M=[m_{ij}]_{0\leq i,j\leq p-1}$ in the second class, we will derive a more precise estimate of \eqref{e:ratebound2u}, by a perturbative argument. Let $\fm_{00}=1-\delta_{00}$, where $\delta_{00}\leq \delta/p$.
We recall the decomposition $\{0,1,2,\cdots, p-1\}=I_1\cup I_2\cup I_3$ from \eqref{e:decompose}. If $0\in I_3$, then 
\begin{align*}
\fn_0\log \sum_{j=1}^{p^{d-1}}\prod_{k=0}^{p-1}\left(\frac{\fm_{0k}}{\fn_0}\right)^{\frac{d-1}{d}\bmw_j(k)}
\leq -\frac{(1-\delta_{00})c(\varepsilon)}{p}.
\end{align*}
Otherwise, $0\in I_2$. Similar argument as in the proof of Proposition \ref{p:ldpboundu}, we have
\begin{align}\label{e:i2i31}
\sum_{i\in I_2\cup I_3}
\fn_i\log \sum_{j=1}^{p^{d-1}}\prod_{k=0}^{p-1}\left(\frac{\fm_{ik}}{\fn_i}\right)^{\frac{d-1}{d}\bmw_j(k)}
\leq -(1+\oo(1))(d-1)\sum_{i\in I_2}\sum_{1\leq k\leq p-1}\fm_{ik}-\frac{c(\varepsilon)}{p}\sum_{i\in I_3}\fn_i.
\end{align}
Moreover, by the definition of the set $I_1$, we have $\sum_{i\in I_1}\fn_i\leq p/(1-\varepsilon)\sum_{i\in I_1}\fm_{i0}=p/(1-\varepsilon)\sum_{i\in I_1}\fm_{0i}$. Therefore,
\begin{align}\begin{split}\label{e:i2i32}
\delta_{00}=\sum_{0\leq i,j\leq p-1\atop (i,j)\neq (0,0)}\fm_{ij}
&\leq\sum_{i\in I_2\atop 1\leq k\leq p-1}\fm_{ik}+\sum_{i\in I_2\setminus\{0\}}\fm_{i0}+\frac{p}{1-\varepsilon}\sum_{i\in I_1}\fm_{0i}+\sum_{i\in I_3}\fn_i\\
&\leq \left(1+\frac{p}{1-\varepsilon}\right)\sum_{i\in I_2 \atop 1\leq k\leq p-1}\fm_{ik}+\sum_{i\in I_3}\fn_i.
\end{split}\end{align}
It follows from combining \eqref{e:i2i31} and \eqref{e:i2i32} we get
\begin{align*}
\sum_{i\in I_2\cup I_3}
\fn_i\log \sum_{j=1}^{p^{d-1}}\prod_{k=0}^{p-1}\left(\frac{\fm_{ik}}{\fn_i}\right)^{\frac{d-1}{d}\bmw_j(k)}
\leq -\frac{c(\varepsilon)\delta_{00}}{p}.
\end{align*}
Thus, the total contribution of terms in \eqref{e:noneq} satisfying $\fb p^3\ln n/n<|m_{00}/(dn)-1|\leq \delta/p$ is bounded by
\begin{align}\label{e:case2}
\exp\left\{-\fb c(\varepsilon)p^2\ln n+\OO(p^2\ln n)\right\}=\frac{\oo(1)}{n^{(d-2)}},
\end{align}
provided that we take $\fb$ sufficiently large.

For the third class, $|m_{00}/(dn)-1|\leq \fb p^3\ln n/n$. We denote 
$\cM(r,\ell)\subset \cM$ the set of $p\times p$ symmetric matrices $M=[m_{ij}]_{0\leq i,j\leq p-1}$ such that 
$\ell=nd-m_{00}$, and $r=n-n_0(M)$. Then $\sum_{i=1}^{p-1}m_{0i}=\sum_{i=1}^{p-1} m_{i0}=\ell-dr$, and $\sum_{1\leq i,j\leq p-1}m_{ij}=2dr-\ell$. Especially, we have $\cM(r,\ell)$ is nonempty only if $dr\leq \ell\leq 2dr$.
The total contribution of terms in \eqref{e:noneq} satisfying $|m_{00}/(dn)-1|\leq \fb p^3 \ln n/n$ is bounded by
\begin{align}\begin{split}\label{e:expldp3}
&\frac{2^{nd/2}(nd/2)!}{(nd)!}\sum_{2\leq r\leq \fb p^3\ln n/d\atop dr\leq\ell\leq 2dr}\sum_{M\in \cM(r,\ell)}{n\choose n_0(M),n_1(M),\cdots, n_{p-1}(M)}\prod_{0\leq i<j\leq p-1}m_{ij}!\prod_{i=0}^{p-1}\frac{m_{ii}!}{2^{m_{ii}/2}(m_{ii}/2)!}\\
&\times\prod_{i=0}^{p-1}|\{(\bmu_1,\bmu_2\cdots, \bmu_{n_i(M)})\in \cU_{d,p}^{n_i(M)}: \bmu_1+\bmu_2+\cdots+\bmu_{n_i(M)}=(m_{i0}, m_{i1},\cdots, m_{ip-1})\}|.
\end{split}\end{align}
We reestimate the first factor on the righthand side of \eqref{e:expldp3},
\begin{align}\begin{split}\label{e:expldp5}
&\phantom{{}={}}\frac{2^{nd/2}(nd/2)!}{(nd)!}{n\choose n_0(M),n_1(M),\cdots, n_{p-1}(M)}\prod_{0\leq i<j\leq p-1}m_{ij}!\prod_{i=0}^{p-1}\frac{m_{ii}!}{2^{m_{ii}/2}(m_{ii}/2)!}\\
&\leq\frac{e^{\OO(r)}}{n^{\ell/2-r}}
\frac{1}{n_1(M)!n_2(M)!\cdots n_{p-1}(M)!}\prod_{0\leq i<j\leq p-1}m_{ij}!\prod_{i=1}^{p-1}\frac{m_{ii}!}{2^{m_{ii}/2}(m_{ii}/2)!}.
\end{split}\end{align}
For the number of walk paths in \eqref{e:expldp3}, we notice that $\bmw_j(1)+\bmw_j(2)+\cdots \bmw_j(p-1)\geq 2$ for $2\leq j\leq p^{d-1}$. Moreover, we have $\bmu_1+\bmu_2+\cdots+\bmu_{n_0(M)}=(m_{00}, m_{01},\cdots, m_{0p-1})$, with $m_{01}+m_{02}+\cdots+m_{0p-1}=dn_{0}(M)-m_{00}=\ell-dr$. Therefore $\bmu_i=\bmw_1$ for all $1\leq i\leq n_0(M)$, except for at most $(\ell-dr)/2$ of them. Therefore, we have
\begin{align}\begin{split}\label{e:pbound1}
&\phantom{{}={}}|\{(\bmu_1,\bmu_2\cdots, \bmu_{n_0(M)})\in \cU_{d,p}^{n_0(M)}: \bmu_1+\bmu_2+\cdots+\bmu_{n_0(M)}=(m_{00}, m_{01},\cdots, m_{0p-1})\}|\\
&\leq \frac{(\ell-dr)!}{m_{01}! m_{02}!\cdots m_{0p-1}!} \sum_{k=1}^{(\ell-dr)/2}n^k{dm-k-1\choose k-1}\leq\OO(1)\frac{(\ell-dr)!}{m_{01}! m_{02}!\cdots m_{0p-1}!} n^{(\ell-dr)/2}.
\end{split}\end{align}
For the number of walk paths in \eqref{e:expldp3} corresponding to $i=1,2,\cdots, p-1$, we have the trivial bound
\begin{align}\begin{split}\label{e:pbound2}
&\phantom{{}={}}|\{(\bmu_1,\bmu_2\cdots, \bmu_{n_i(M)})\in \cU_{d,p}^{n_i(M)}: \bmu_1+\bmu_2+\cdots+\bmu_{n_i(M)}=(m_{i0}, m_{i1},\cdots, m_{ip-1})\}|\\
&\leq \frac{(dn_i(M))!}{m_{i0}! m_{i1}!\cdots m_{ip-1}!}.
\end{split}\end{align}
Combining the estimates \eqref{e:pbound1} and \eqref{e:pbound2}, we get the following bound on the number of walk paths in \eqref{e:expldp3},
\begin{align}\begin{split}\label{e:expldp6}
&\phantom{{}={}}\prod_{i=0}^{p-1}|\{(\bmu_1,\bmu_2\cdots, \bmu_{n_i(M)})\in \cU_{d,p}^{n_i(M)}: \bmu_1+\bmu_2+\cdots+\bmu_{n_i(M)}=(m_{i0}, m_{i1},\cdots, m_{ip-1})\}|\\
&\leq\OO(1) n^{(\ell-dr)/2}\frac{(\ell-dk)!}{m_{01}! m_{02}!\cdots m_{0p-1}!}
\prod_{i=1}^{p-1}\frac{(dn_i(M))!}{m_{i0}! m_{i1}!\cdots m_{ip-1}!}
.
\end{split}\end{align}
The total contribution of terms in \eqref{e:noneq} satisfying $|m_{00}/(dn)-1|\leq \fb p^3\ln n/n$ is bounded by
\begin{align}\begin{split}\label{e:case3}
&\phantom{{}={}}\sum_{2\leq r\leq \fb p^3\ln n/d\atop dr\leq\ell\leq 2dr}\sum_{M\in \cM(r,\ell)}
\frac{e^{\OO(r)}(\ell-dr)!}{n^{(d/2-1)r}}
\prod_{i=1}^{p-1}\frac{(dn_i(M))!}{n_i(M)!}
\prod_{0\leq i<j\leq p-1}\frac{1}{m_{ij}!}\prod_{i=1}^{p-1}\frac{1}{2^{m_{ii}/2}(m_{ii}/2)!}\\
&\leq\sum_{2\leq r\leq \fb p^3\ln n/d\atop dr\leq\ell\leq 2dr}
\frac{e^{\OO(r)}(\ell-dr)!r^{(d-1)r}}{n^{(d/2-1)r}}
\sum_{M\in \cM(r,\ell)}\prod_{i=1}^{p-1}\frac{1}{m_{i0}!}
\prod_{1\leq i\leq j\leq p-1}\frac{1}{(m_{ij}/(1+\delta_{ij}))!}\\
&\leq \sum_{2\leq r\leq \fb p^3\ln n/d\atop dr\leq \ell\leq 2dr }
\frac{e^{\OO(r)}r^{(d-1)r}}{n^{(d/2-1)r}}
\frac{p^{dr}}{(dr-\ell/2)!}\\
&\leq 
\sum_{2\leq r\leq \fb p^3\ln n/d}
\left(\frac{\OO(1)p^d r^{(d/2-1)}}{n^{(d/2-1)}}\right)^r\leq \frac{\OO(1)p^{2d}}{n^{d-2}},
\end{split}\end{align}
provided that $p\ll n^{(d-2)/(5d-6)}$.

For the last class, the total number of such $p\times p$ symmetric matrices is $e^{\OO(p^2\ln n)}$, and thanks to Proposition \ref{p:ldpboundu}, each term is exponentially small, i.e. $e^{-c(\delta)n/p^2}$. Therefore the total contribution is
\begin{align}\label{e:case4}
\exp\{-c(\delta) n/p^2+\OO(p^2\ln n)\}=\frac{\oo(1)}{n^{(d-2)}}.
\end{align}

The claim \eqref{e:noneq} follows from combining the discussion of all four cases, \eqref{e:case1}, \eqref{e:case2}, \eqref{e:case3} and \eqref{e:case4}.
\end{proof}
\bibliography{References.bib}{}
\bibliographystyle{abbrv}

\end{document}